\documentclass[11pt, reqno, english]{amsart}
\usepackage{style}
\usepackage[normalem]{ulem}
\usepackage{array} 
\usepackage{caption} 
\usepackage{tikz}
\usetikzlibrary{calc}
\definecolor{lightblue}{RGB}{20,147,255}
\usetikzlibrary{math} 
\usetikzlibrary{shapes.geometric, calc} 

\begin{document}

\title[Gromov--Hausdorff distances, Borsuk--Ulam theorems, and Vietoris--Rips complexes]{Gromov--Hausdorff distances, Borsuk--Ulam theorems,\\ and Vietoris--Rips complexes}

\pagestyle{plain}

\author{Henry Adams}
\address[HA]{Department of Mathematics, University of Florida, Gainesville, FL 32611, USA}
\email{henry.adams@ufl.edu}

\author{Johnathan Bush}
\address[JB]{Department of Mathematics and Statistics, James Madison University, Harrisonburg, VA 22807, USA}
\email{bush3je@jmu.edu}

\author{Nate Clause}
\address[NC]{Department of Mathematics,
The Ohio State University, Columbus, OH 43202, USA}
\email{clause.15@osu.edu}

\author{Florian Frick}
\address[FF]{Dept.\ Math.\ Sciences, Carnegie Mellon University, Pittsburgh, PA 15213, USA}
\email{frick@cmu.edu} 

\author{Mario G\'{o}mez}
\address[MG]{Department of Mathematics,
The Ohio State University, Columbus, OH 43202, USA}
\email{gomezflores.1@osu.edu}

\author{Michael Harrison}
\address[MH]{Institute for Advanced Study, Princeton, NJ 08540, USA}
\email{mah5044@gmail.com}

\author{R.~Amzi Jeffs}
\address[RAJ]{Dept.\ Math.\ Sciences, Carnegie Mellon University, Pittsburgh, PA 15213, USA}
\email{amzijeffs0@gmail.com}

\author{Evgeniya Lagoda}
\address[EL]{Institut f\"ur Mathematik, Freie Universit\"at Berlin, 14195 Berlin, Germany}
\email{e.lagoda@fu-berlin.de}

\author{Sunhyuk Lim}
\address[SL]{Sungkyunkwan University, 16419 Suwon-si, Gyeonggi-do, Republic of Korea}
\email{lsh3109@skku.edu}

\author{Facundo M\'emoli}
\address[FM]{Department of Mathematics,
Rutgers University, Piscataway, NJ 08854, USA}
\email{facundo.memoli@rutgers.edu}

\author{Michael Moy}
\address[MM]{Energy Institute, Colorado State University, Fort Collins, CO 80524, USA}
\email{michael.moy@colostate.edu}

\author{Nikola Sadovek}
\address[NS]{Institut f\"ur Mathematik, Freie Universit\"at Berlin, 14195 Berlin, Germany}
\email{nikola.sadovek@fu-berlin.de}

\author{Matt Superdock}
\address[MS]{Dept.\ Mathematics and Computer Science, Rhodes College, Memphis, TN 38112, USA}
\email{superdockm@rhodes.edu}

\author{Daniel Vargas-Rosario}
\address[DVR]{Department of Mathematics, University of Colorado, Boulder, CO 80309, USA}
\email{Daniel.Vargas-Rosario@colorado.edu}

\author{Qingsong Wang}
\address[QW]{Hal{\i}c{\i}o\u{g}lu Data Science Institute, University of California San Diego, CA 92093, USA}
\email{qswang92@gmail.com}

\author{Ling Zhou}
\address[LZ]{Department of Mathematics,
Department of Mathematics, Duke University, Durham, NC 27710, USA}
\email{ling.zhou@duke.edu}



\thanks{
This paper is the result of a polymath-style collaboration.
JB was supported by the NSF-Simons Southeast Center for Mathematics and Biology through NSF grant DMS 1764406 and Simons Foundation grant 594594.
FF was supported by NSF grant DMS 1855591, NSF CAREER Grant DMS 2042428, and a Sloan Research Fellowship.
MH was supported by the Institute for Advanced Study through the NSF Grant DMS 1926686.
RAJ was supported by NSF grant DMS 2103206.
FM was supported by  NSF grants DMS 1547357, CCF 1740761 and IIS 1901360, and also by BSF grant 2020124.
NS was funded by the Deutsche Forschungsgemeinschaft (DFG, German Research Foundation) under Germany's Excellence Strategy - The Berlin Mathematics Research Center MATH+ (EXC-2046/1, project ID 390685689, BMS Stipend).
}

\begin{abstract}
We explore emerging relationships between the Gromov--Hausdorff distance, Borsuk--Ulam theorems, and Vietoris--Rips simplicial complexes.
The Gromov--Hausdorff distance between two metric spaces $X$ and~$Y$ can be lower bounded by the distortion of (possibly discontinuous) functions between them.
The more these functions must distort the metrics, the larger the Gromov--Hausdorff distance must be.
Topology has few tools to obstruct the existence of discontinuous functions.
However, an arbitrary function $f\colon X\to Y$ induces a continuous map between their Vietoris--Rips simplicial complexes, where the allowable choices of scale parameters depend on how much the function $f$ distorts distances.
We can then use equivariant topology to obstruct the existence of certain continuous maps between Vietoris--Rips complexes.
With these ideas we bound how discontinuous an odd map between spheres $S^k\to S^n$ with $k>n$ must be, generalizing a result by Dubins and Schwarz (1981), which is the case $k=n+1$.
As an application, we recover or improve upon all of the lower bounds from Lim, M{\'e}moli, and Smith (2022) on the Gromov--Hausdorff distances between spheres of different dimensions.
We also provide new upper bounds on the Gromov--Hausdorff distance between spheres of adjacent dimensions.
\end{abstract}

\maketitle

\section{Introduction}
\label{sec:intro}
The Gromov--Hausdorff distance between metric spaces $X$ and $Y$, denoted by $d_{\gh}(X,Y)$, quantifies the extent to which $X$ and $Y$ fail to be isometric.
The Gromov--Hausdorff distance is used in many areas of geometry~\cite{BuragoBuragoIvanov,cheeger1997structure,colding1996large,petersen2006riemannian}.
In applications to shape and data comparison/classification, one desires to estimate either the Gromov--Hausdorff distance between spaces~\cite{ms04,ms05,memoli2007use} or the Gromov--Wasserstein distance~\cite{memoli2011gromov,sturm2012space,peyre2019computational,alvarez2018gromov}, which is one of its optimal transport induced variants.
However, both distances are hard to compute, both analytically and algorithmically~\cite{memoli2012some,schmiedl2015shape,schmiedl2017computational,agarwal2018computing}.
Despite the interest in this type of distances, exact values of the Gromov--Hausdorff distance are known in only a small number of cases; see Section~\ref{sec:related}.

Our paper is the result of a polymath-style collaboration, which began as an attempt to explain the following motivating question. 
In~\cite{lim2023gromov}, Lim, M\'{e}moli, and Smith prove the first strong bounds for the Gromov--Hausdorff distance between spheres of different dimensions.
We were surprised to observe that the values in~\cite{lim2023gromov} (see Table~\ref{table:gh}) had recently appeared in the literature in a different context, namely as the scale parameters when Vietoris--Rips complexes of spheres change homotopy type.
The \emph{Vietoris--Rips complex} $\vr{X}{r}$, which coarsens a metric space $X$ with respect to some scale parameter $r \geq 0$, is commonly used in applied topology to approximate the shape of a dataset~\cite{Carlsson2009}, and has its historical origins in algebraic topology~\cite{Vietoris27} and geometric group theory~\cite{Gromov}.
The lower bound $2\cdot d_{\gh}(S^n,S^{n+1})\ge r_n$ from~\cite{lim2023gromov} reminded us of the fact that the first change in homotopy type of $\vr{S^n}{r}$ occurs when $r=r_n$~\cite{AAF,lim2020vietoris,katz1989diameter}, where $r_n$ is the geodesic distance between two vertices of the regular $(n+1)$-simplex inscribed in $S^n$.
Similarly, the equality $2\cdot d_{\gh}(S^1,S^2) = 2\cdot d_{\gh}(S^1,S^3)= r_1 = \frac{2\pi}{3}$ from~\cite{lim2023gromov} reminded us of the homotopy equivalence $\vr{S^1}{\frac{2\pi}{3}+\varepsilon}\simeq S^3$ from~\cite{AA-VRS1,ABF,moy2023vietoris}. 

\begin{question-motivating}
What is the connection between the Gromov--Hausdorff distance between spheres and Vietoris--Rips complexes of spheres?
\end{question-motivating}

To provide an answer to this question\footnote{See Section~\ref{sec:related} for other connections, including the stability of persistent homology.}, we combine and extend two generalizations of the Borsuk--Ulam theorem: one by Dubins and Schwarz~\cite{dubins1981equidiscontinuity} used in Lim, M\'{e}moli and Smith~\cite{lim2023gromov} to study the Gromov--Hausdorff distance, and the second by Adams, Bush, and Frick~\cite{ABF,ABF2} on the equivariant topology of Vietoris--Rips complexes.
The Borsuk--Ulam theorem, a classic result in equivariant topology, states that there is no continuous $\Z/2$ equivariant map $f \colon S^k \to S^n$ for $k > n$~\cite{matousek2003using}.
Here the $\Z/2$ action on each sphere is the antipodal map, and $f$ is called \emph{$\Z/2$ equivariant} or \emph{odd} if it commutes with the $\Z/2$ actions; that is, if $f(-x) = -f(x)$ for all $x \in S^k$.

We relate Gromov--Hausdorff distances, Borsuk--Ulam theorems, and Vietoris--Rips complexes as follows.
Estimating the Gromov--Hausdorff distance $d_{\gh}(X,Y)$ involves bounding the \emph{distortion} $\dis(f)$ of a (possibly discontinuous) function $f:X\to Y$, which measures the extent to which $f$ fails to preserve distances: the more that functions between $X$ and $Y$ must distort the metrics, the larger $d_{\gh}(X,Y)$ must be.
When $X$ and $Y$ are spheres, Lim, M\'{e}moli and Smith~\cite{lim2023gromov} show that it suffices to consider odd functions; this is the so-called ``helmet trick''.
We transform an odd function $f\colon S^k \to S^n$ into a \emph{continuous} odd map $|\vr{S^k}{r}| \to |\vr{S^n}{r+\dis(f)}|$ for any $r \geq 0$, letting the Vietoris--Rips complexes absorb discontinuities.
We then obstruct the existence of such maps with the equivariant topology of Vietoris--Rips complexes, measured via the following quantity.

\begin{definition}
\label{def:cnk}
For $k\ge n$, we define
\[c_{n,k} \coloneqq \inf\{r\ge 0 \mid \text{there exists an odd map }S^k \to \vr{S^n}{r}\}.\]
\end{definition}

Due to a theorem of Hausmann~\cite{Hausmann1995}, we have a homotopy equivalence $\vr{S^n}{r} \simeq S^n$ for sufficiently small $r$, and moreover there is an odd map $\vr{S^n}{r} \to S^n$.
The Borsuk--Ulam theorem then implies that no odd map $S^k \to \vr{S^n}{r}$ exists for such $r$ unless $k \le n$.
In particular, $c_{n,n}=0$, but $c_{n,k} > 0$ for $k > n$.
Therefore, intuitively, the quantity $c_{n,k}$ represents the amount by which $S^n$ needs to be ``thickened'' until it admits an odd map from $S^k$.

\smallskip
Our main result is the following lower bound on $d_{\gh}(S^n,S^k)$.

\begin{table}
\def\arraystretch{1.2}
\begin{tabular}{| >{ $} c <{$} | >{\centering $} m{.08\textwidth} <{$} | >{\centering $} m{.08\textwidth} <{$} | >{\centering $} m{.08\textwidth} <{$} | >{ $} c <{$} | >{ $} c <{$} | >{ $} c <{$} | >{ $} c <{$} |}
\hline
_n \backslash ^k & 1 & 2 & 3 & 4 & 5 & 6 & 7 \\
\hline
1 & 0 & \frac{2\pi}{3} & \frac{2\pi}{3} & \big[\textcolor{blue}{\frac{4\pi}{5}},\,\pi\big) & \big[\textcolor{blue}{\frac{4\pi}{5}},\,\,\pi\big) & \big[\textcolor{blue}{\frac{6\pi}{7}},\,\,\pi\big) & \big[\textcolor{blue}{\frac{6\pi}{7}},\,\,\pi\big) \\
\hline
2 & & 0 & r_2 & \big[r_2,\,\,\pi\big) & \big[\textcolor{blue}{ c_{2,5}},\pi\big) & \big[\textcolor{blue}{ c_{2,6}},\pi\big) & \big[\textcolor{blue}{ c_{2,7}},\pi\big) \\
\hline
3 & & & 0 & \big[r_3,\textcolor{blue}{\frac{2\pi}{3}}\big] & \big[r_3,\,\,\,\pi\big) & \big[\textcolor{blue}{c_{3,6}},\pi\big) & \big[\textcolor{blue}{c_{3,7}},\pi\big) \\
\hline
4 & & & & 0 & \big[r_4,\textcolor{blue}{\frac{2\pi}{3}}\big] & \big[r_4,\,\,\,\pi\big) & \big[\textcolor{blue}{c_{4,7}},\pi\big) \\
\hline
5 & & & & & 0 & \big[r_5,\textcolor{blue}{\frac{2\pi}{3}}\big] & \big[r_5,\,\,\,\pi\big) \\
\hline
6 & & & & & & 0 & \big[r_6,\textcolor{blue}{\frac{2\pi}{3}}\big] \\
\hline
7 & & & & & & & 0 \\
\hline
\end{tabular}
\begin{tikzcd}[row sep=tiny]
	\begin{tikzpicture}[scale=.7]
    \filldraw[fill=none](0,0) circle (1);
    \draw (0,1)--(.86,-.5)--(-.86,-.5)--(0,1);
    \node at (0, -.85)  {$r_1$};
	\end{tikzpicture}
	\hspace{5mm}
	\begin{tikzpicture}[scale=.7]
    \filldraw[fill=none](0,0) circle (1);
    \draw (0,1)--(.86,-.5)--(-.86,-.5)--(0,1);
    \node at (0, -.85)  {$r_2$};
    \draw [dashed] (0,0)--(.86,-.5);
    \draw [dashed] (0,0)--(-.86,-.5);
    \draw [dashed] (0,0)--(0,1);
	\end{tikzpicture}
	\hspace{5mm}
	\hdots
	\\
   r_1=\frac{2\pi}{3}\hspace{30mm}
\end{tikzcd}
\caption{
Bounds for the quantity $2\cdot d_{\gh}(S^n, S^k)$ for small values of $n$ and $k$.
Here $r_n = \arccos\left(\tfrac{-1}{n+1}\right)$ and $c_{n,k}=\inf\{r\ge 0 \mid \exists\text{ an odd map }S^k \to \vr{S^n}{r}\}$.
The entries in black appear in~\cite[Figure~2]{lim2023gromov}, and the entries in blue are new.
Our Main Theorem 
recovers or improves upon all known lower bounds, and the upper bound by $\frac{2\pi}{3}$ along the superdiagonal is established in Theorem~\ref{thm:super-diag-upper-bound}.
}
\label{table:gh}
\end{table}

\begin{theorem-main}
For all $k \ge n$, the following inequalities hold:
\begin{align*}
2\cdot d_{\gh}(S^n, S^k)
&\ge \inf \left\{ \dis(f) \mid f \colon S^k \to S^n \textnormal{ is odd}\right\} \\
&\ge \inf \left\{r\ge 0 \mid \exists \textnormal{ odd }S^k \to \vr{S^n}{r}\right\} \eqqcolon c_{n,k}.
\end{align*}
\end{theorem-main}

Let us explain the two inequalities in our Main Theorem, and compare them to existing results.
The first inequality in our Main Theorem is the aforementioned ``helmet trick'' by Lim, M\'{e}moli, and Smith~\cite[Lemma~5.5]{lim2023gromov}, which states that to bound $d_{\gh}(S^n,S^k)$, it is enough to consider the distortion of odd functions.
Lim, M\'{e}moli, and Smith then prove that the distortion of any such function is bounded below by $r_n$, using Dubins and Schwarz's generalization of the Borsuk--Ulam theorem mentioned above~\cite{dubins1981equidiscontinuity} (more details will be described below).
This implies $2\cdot d_{\gh}(S^n,S^k)\geq r_n$ for all $k>n$.
Combining this with explicit constructions of upper bounds, they proved that $2\cdot d_{\gh}(S^n,S^k)=r_n$ holds exactly for $n < k\le 3$ and gave nontrivial bounds in all dimensions.
Despite these tight results, one  unsatisfactory feature of the general lower bound $2\cdot d_{\gh}(S^n,S^k)\geq r_n$ for $k>n$ is that the right-hand side does not depend on $k$, whereas it is known that for $n$ fixed and $k\to \infty$, $2\cdot d_{\gh}(S^n,S^k)\to \pi > r_n$~\cite[Proposition 1.8]{lim2023gromov}.
The present paper establishes lower bounds which improve upon these.

The second inequality in our Main Theorem, which we prove in Section~\ref{sec:main-theorem-proof}, lower bounds the distortion of an odd map $S^k\to S^n$ with $k\ge n$ in terms of the equivariant topology of Vietoris--Rips complexes of spheres.
The motivation for studying odd maps $S^k \to \vr{S^n}{r}$ comes from Adams, Bush, and Frick~\cite{ABF2}, who observe the following.
Even though we do not have a complete understanding how the homotopy types of $\vr{S^n}{r}$ change as the scale parameter $r$ increases, we can control the equivariant topology of $\vr{S^n}{r}$ in terms of packings and coverings in projective space.
In particular, if there exists a sufficiently efficient covering of $\RP^n$ by $k$ points, then there does not exist an odd map $S^k \to \vr{S^n}{r}$~\cite{ABF2}, which allows us in Section~\ref{sec:known-values-cnk} to estimate the quantity $c_{n,k}$ in terms of the covering number of $k$ points in $\RP^n$.
In this same section we furthermore determine some values of $c_{n,k}$ exactly using the current limited understanding of the homotopy types of $\vr{S^n}{r}$.
When combined together, these estimates show that the lower bound $2\cdot d_{\gh}(S^n,S^k)\ge c_{n,k}$ from our Main Theorem is never worse (and frequently improves upon) those from~\cite{lim2023gromov}; see Remark~\ref{rem:improvement} and Table~\ref{table:gh}.
In Section~\ref{sec:super-diag-upper-bound}, we supplement these new lower bounds with the following  upper bounds when $k = n +1$ (which improve upon those from~\cite{lim2023gromov}).

\begin{theorem}
\label{thm:super-diag-upper-bound}
For every $n\ge 1$, we have $2\cdot d_{\gh}(S^n, S^{n+1}) \le \frac{2\pi}{3}$.
\end{theorem}

The second inequality of the Main Theorem is of independent interest due to its relationship with the following natural question: the Borsuk--Ulam theorem asserts that there exists no continuous odd map from $S^k \to S^n$ for $k > n$, so given an odd function from $S^k \to S^n$, how discontinuous must it be?
In~\cite{dubins1981equidiscontinuity}, Dubins and Schwarz quantify the discontinuity of odd functions $S^{n+1} \to S^n$ by showing that the modulus of discontinuity of any odd function $S^{n+1}\to S^n$ is at least $r_n$.
Moreover, they exhibit a function which realizes this bound.
In Section~\ref{sec:gen-ds} we generalize the Dubins--Schwarz inequality, by adapting the proof of the second inequality in the Main Theorem to use the modulus of discontinuity instead of the distortion.

\begin{theorem}[Generalized Dubins--Schwarz inequality]
\label{thm:odd-modulus-discontinuity-bound}
Any odd function $f\colon S^k \to S^n$ with $k \ge n$ has modulus of discontinuity at least $c_{n,k}$, and this bound is tight.
In particular, for every $\varepsilon > 0$, there exists an odd function $S^k \to S^n$ with modulus of discontinuity $c_{n,k} + \varepsilon$.
\end{theorem}

In summary, this paper explores and combines emerging relationships between the Gromov--Hausdorff distance, Borsuk--Ulam theorems, and Vietoris--Rips simplicial complexes.
While it was previously known that these topics were pairwise related (see Figure~\ref{fig:triangle-graphic} and Section~\ref{sec:related}), our Main Theorem exhibits an explicit mutual connection between these concepts.

Researchers from different research communities, such as applied topology, topological combinatorics, geometric group theory, metric geometry, and quantitative topology, have different perspectives and levels of expertise on the Gromov--Hausdorff distance, Borsuk--Ulam theorems, and Vietoris--Rips simplicial complexes.
There may be few experts on all three topics.
As such, we include a thorough survey of these topics and the existing relationships among them in Sections~\ref{sec:related} and~\ref{sec:background},  in hopes that this paper will serve as an efficient way to teach these topics to a variety of research communities.
Additionally, we have collected a large number of remaining open questions in Section~\ref{sec:conclusion}, many of which we hope will yield to multi-pronged attacks, after bridges have been formed between these different communities.

\section{Related work}
\label{sec:related}

\begin{figure}[htb]
\centering
\includegraphics[width=\textwidth]{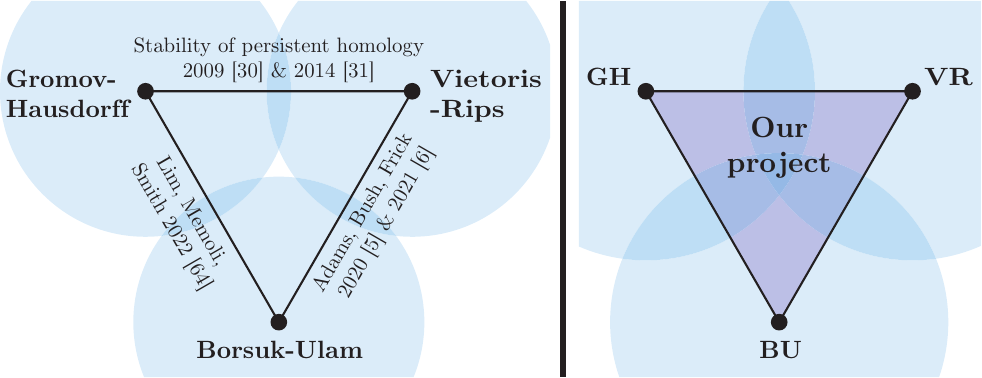}
\caption{Our project fills a hole in the mathematical landscape.
See the open questions in Section~\ref{sec:conclusion} for areas where more work is needed.}
\label{fig:triangle-graphic}
\end{figure}

We organize our description of related work using Figure~\ref{fig:triangle-graphic}.

\subsection*{The Gromov--Hausdorff (GH) distance}

The Gromov--Hausdorff distance provides a metric on isometry classes of compact metric spaces~\cite{edwards1975structure,gromov1981groups,gromov1981structures,tuzhilin2016invented}.
Despite its importance in geometry~\cite{BuragoBuragoIvanov,cheeger1997structure,colding1996large,petersen2006riemannian} and shape comparison~\cite{ms04,ms05,memoli2007use}, exact Gromov--Hausdorff distances are only known in a small number of cases.
These include the Gromov--Hausdorff distance between a line segment and a Euclidean circle~\cite{ji2021gromov}, between spheres of dimension at most three~\cite{lim2023gromov}, and between some pairs of discrete metric spaces such as simplices~\cite{memoli2012some,ivanov2019gromov} and between the vertex sets of regular polygons~\cite{lim2023gromov,talipov2023gromov}.

\subsection*{Vietoris--Rips (VR) complexes}

Vietoris--Rips simplicial complexes were first considered by Vietoris in the context of developing a cohomology theory for metric spaces~\cite{lefschetz1942algebraic,Vietoris27}, and introduced independently by Rips in geometric group theory as a natural way to thicken (i.e.\ coarsen) a space~\cite{bridson2011metric,Gromov}.
More recently, they have become commonly used tools in applied and computational topology~\cite{edelsbrunner2000topological,EdelsbrunnerHarer}, used in applications to
data analysis~\cite{Carlsson2009,CarlssonIshkhanovDeSilvaZomorodian2008,ghrist2008barcodes} and sensor networks~\cite{Coordinate-free,de2007coverage}, for example.

\subsection*{Borsuk--Ulam (BU) theorems}

The Borsuk--Ulam theorem is a classic result from topology, stating that any continuous map from the $n$-sphere to $n$-dimensional Euclidean space identifies antipodal points, or equivalently that there is no continuous odd map from the $k$-sphere to the $n$-sphere for $k>n$.
It has numerous applications to discrete geometry and combinatorics~\cite{matousek2003using}, many of which are still being discovered and explored, such as applications to low-distortion embeddings of finite metric spaces into Euclidean space~\cite{sidiropoulos2019} (here the Borsuk--Ulam theorem is used in a different fashion to our approach of bounding distortion), inscribing parallelograms into spatial curves~\cite{aslam2020}, and hardness results for graph colorings~\cite{austrin-bhangale-potukuchi2020}.
Various recent applications of equivariant topology go beyond the antipodal symmetry of the Borsuk--Ulam theorem; see~\cite{blagojevic2017}.

\subsection*{VR--GH}

A well-known connection between Vietoris--Rips complexes and Gromov--Hausdorff distances is the stability of persistent homology:
If $X$ and $Y$ are totally bounded metric spaces, then twice the Gromov--Hausdorff distance between $X$ and $Y$ is bounded from below by the bottleneck distance between the Vietoris--Rips persistent homology barcodes of $X$ and $Y$~\cite{ChazalDeSilvaOudot2014,chazal2009gromov,cohen2007stability}.
However, stability alone does not provide sharp lower bounds on the Gromov--Hausdorff distances between spheres of different dimensions. In fact, those lower bounds have been computed exactly in~\cite[Corollary 9.3]{lim2020vietoris} where it is proved that they equal $\frac{1}{2}$ of the filling radius~\cite{gromov1983filling,katz1983filling} of the sphere with smaller dimension.
In the cases $(n,k)\in\{(1,2),(1,3),(2,3)\}$, these bounds  yield exactly one-half of the actual corresponding values  of the Gromov-Hausdorff distances~\cite{lim2023gromov}.
We show how to inject ideas from equivariant topology into the VR--GH story so as to obtain sharper bounds.

\subsection*{GH--BU}

The paper~\cite{lim2023gromov} computes $d_{\gh}(S^n,S^k)$ exactly for $n < k\le 3$ and gives nontrivial upper and lower bounds for all $n < k$.
Some of their lower bounds are related to generalizations of the Borsuk--Ulam theorem, such as~\cite{dubins1981equidiscontinuity}, and these lower bounds strictly improve upon the lower bounds provided by the stability of persistent homology.

\subsection*{BU--VR}

The papers~\cite{ABF,ABF2} use information about the homotopy connectivity of Vietoris--Rips complexes defined on spheres at various scales to prove generalizations of the Borsuk--Ulam theorem for maps from spheres into higher-dimensional Euclidean spaces.
Cohomological techniques, without knowledge of the connectivity of Vietoris--Rips complexes, are used in~\cite{crabb2023borsuk} to obtain similar, and sometimes stronger, results.

\subsection*{GH--BU--VR}

The Gromov--Hausdorff distance, Borsuk--Ulam theorems, and Vietoris--Rips complexes are not only related pairwise.
Indeed, we think of our paper as a witness point showing that there is a nontrivial triple intersection between these topics.
For example, our Main Theorem lower bounds the Gromov--Hausdorff distance $d_{\gh}(S^n,S^k)$ for $k\ge n$ in terms of odd maps from from $S^k$ into the Vietoris--Rips complex $\vr{S^n}{r}$, and we obstruct the existence of such odd maps using the Borsuk--Ulam theorem.

\section{Background and notation}
\label{sec:background}

For topological spaces $X$ and $Y$:
\begin{itemize}
\item
A \emph{map} $f \colon X \to Y$ is a continuous function.
\item
A \emph{function} $f \colon X \to Y$ is any function, possibly discontinuous.
\end{itemize}

For a metric space $X$:
\begin{itemize}
\item
We denote by $d_{X} \colon X \times X \to \mathbb{R}$ the metric on $X$.
\item
We let $B(x;r)\coloneqq \{x'\in X~|~d_X(x',x)<r\}$ denote the open ball of radius $r$ about $x$.
For $X'\subseteq X$, we let $B(X';r)=\cup_{x\in X'}B(x;r)$ be the union of the balls.
\item
The \emph{diameter} of a subset $A \subseteq X$ is $\diam(A) \coloneqq \sup_{a, a' \in A} d_X(a, a')$.
\end{itemize}

We define the \emph{$n$-sphere} as $S^{n} \coloneqq \{x \in \R^{n + 1} \mid \|x\| = 1\}$.
We always equip $S^n$ with the \emph{geodesic metric} in which great circles have length $2\pi$ (with the exception on Section~\ref{ssec:Euclidean}, when we also consider the Euclidean metric).
For $x,x'\in S^n\subset \R^{n+1}$, the geodesic metric satisfies the equality
\[d_{S^n}(x,x')=\arccos\left(\langle x,x'\rangle\right)=2\arcsin\left(\tfrac{\|x-x'\|}{2}\right).\]
With this convention we have $\diam(S^0)=\pi$, even though $S^0$ is not a geodesic space.

\subsection{Background on the Gromov--Hausdorff distance}
\label{ssec:background-gh}

\subsection*{Distortion}

Given any two bounded metric spaces $(X,d_X)$ and $(Y,d_Y)$ and any non-empty relation $R\subseteq X\times Y$, the \emph{distortion of $R$} is defined as 
\[\text{dis}(R)\coloneqq \sup_{(x,y),(x',y')\in R}\left|d_X(x,x')-d_Y(y,y')\right|.\]
In particular, the graph of any function $g\colon X\to Y$ is a relation $R_g\subseteq X\times Y$, and we denote the distortion of this relation by $\dis(g)\coloneqq \dis(R_g)$.
In this case, 
\[\dis(g)=\sup_{x,x'\in X}|d_{X}(x,x')-d_{Y}(g(x),g(x'))|.\]
A relation is a \emph{correspondence} if its projections onto $X$ and onto $Y$ are surjective.
Note that the relation $R_g$ is a correspondence if and only if $g$ is surjective.

Given functions $g \colon X \to Y$ and $h \colon Y \to X$ between metric spaces, the \emph{codistortion} (see Figure~\ref{fig:codistortion}) of $g$ and $h$ is defined as
\[\codis(g,h)\coloneqq\sup_{x\in X,y\in Y}|d_{X}(x,h(y))-d_{Y}(g(x),y)|.\]
The codistortion $\codis(g,h)$ allows one to bound the extent to which the functions $g$ and $h$ fail to be inverses of each other.
Indeed, if $\codis(g,h)<\varepsilon$, then one has $d_X(x,h(g(x)))<\varepsilon$ and $d_Y(g(h(y)),y)<\varepsilon$ for every $x\in X$ and $y\in Y$.

\begin{figure}[h]
\centering
\begin{tikzcd}
    \begin{tikzpicture}[scale=0.8]
    \filldraw[fill=none](0,0) circle (1.5);
    \draw [dashed](-1.5,0) to[out=90,in=90,looseness=.5] (1.5,0);
    \draw (1.5,0) to [out=270,in=270,looseness=.5] (-1.5,0);
    \node[fill=blue,circle,inner sep=.1em] (x) at (1.2,0.891) {};
    \node[right,blue] at (x)  {$x$};
    \node[fill=purple,circle,inner sep=.1em] (y) at (1.2,-0.891) {};
    \node[right,purple] at (y)  {$h(y)$};
	\end{tikzpicture}
	\hspace{-4em}
    & 
	X \ar[rr,bend left,"g"] && 
	Y \ar[ll,bend left,"h"]
	&
	\hspace{-2.5em}
	\begin{tikzpicture}[scale=0.8]
    \filldraw[fill=none](0,0) circle (1.5);
    \node[fill=blue,circle,inner sep=.1em] (x1) at (0.95,1.15) {};
    \node[right,blue] at (x1)  {$g(x)$};
    \node[fill=purple,circle,inner sep=.1em] (y1) at (1,-1.1) {};
    \node[right,purple] at (y1)  {$y$};
	\end{tikzpicture}
	\end{tikzcd}
\caption{Illustration of the codistortion.}
\label{fig:codistortion}
\end{figure}
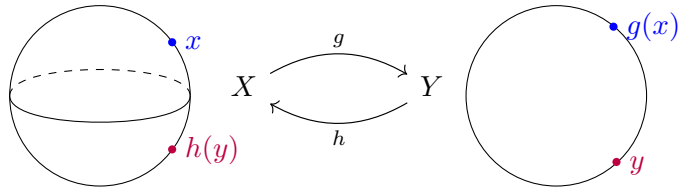

\subsection*{Hausdorff distance}

Let $Z$ be a metric space.
If $X$ and $Y$ are two closed submetric spaces of $Z$ then the \emph{Hausdorff distance} between $X$ and $Y$ is
\[d_\mathrm{H}(X,Y)=\inf\{r\geq 0\mid X\subseteq B(Y;r)\;\mbox{and}\; Y\subseteq B(X;r)\}.\]
In other words, the Hausdorff distance calculates the smallest real number $r$ such that if we thicken $Y$ by $r$ it contains $X$ and if we thicken $X$ by $r$ it contains $Y$.

\begin{figure}[ht]
    \centering
    \includegraphics[width=.85\linewidth]{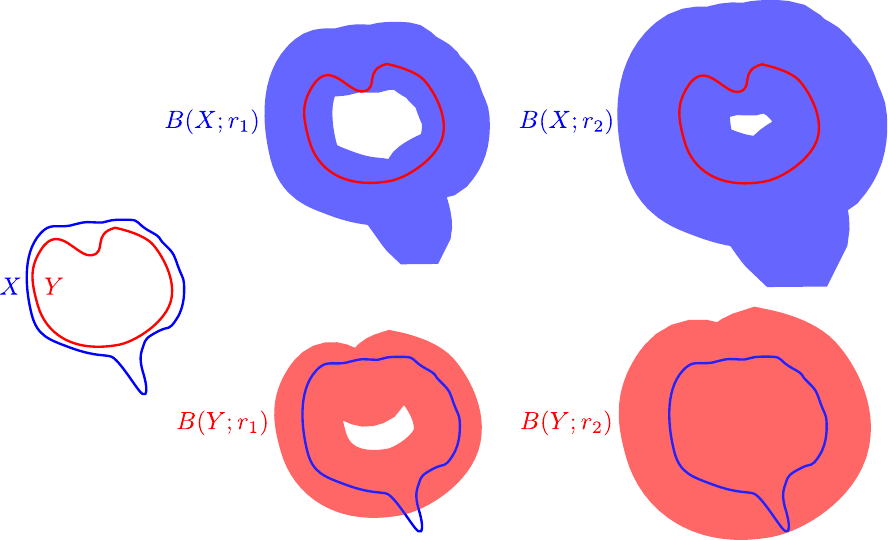}
    \caption{Let the metric spaces $X$ (blue) and $Y$ (red) inherit the Euclidean metric from the plane.
    If we thicken $X$ by $r_1$, then $Y\subseteq B(X;r_1)$, but $X\not\subseteq B(Y;r_1)$, so $d_\mathrm{H}(X,Y)\ge r_1$.
    When thickening each space by  $r_2$, we see $Y\subseteq B(X;r_2)$ and $X\subseteq B(Y;r_2)$, so $d_\mathrm{H}(X,Y)\le r_2$.}
    \label{fig:hausdorff}
\end{figure}

\subsection*{Gromov--Hausdorff distance}

The Gromov--Hausdorff distance $d_{\gh}(X,Y)$ between two bounded metric spaces $X$ and $Y$ is defined as the infimum, over all metric spaces $Z$ and isometric embeddings $\gamma\colon X\to Z$ and $\varphi\colon Y\to Z$, of the Hausdorff distance between $\gamma(X)$ and $\varphi(Y)$~\cite{edwards1975structure,gromov2007metric}.
Unlike the Hausdorff distance, the Gromov-Hausdorff distance considers sets $X$ and $Y$ that are not part of the same metric space.
However, to compute it we need to embed $X$ and $Y$ in different common metric spaces $Z$, and then take the infimum of the Hausdorff distance over those $Z$.
It follows from~\cite{kalton1999distances} that the Gromov--Hausdorff distance between any two bounded metric spaces $X$ and $Y$ can alternatively be defined as
\[2\cdot d_{\gh}(X,Y) = \inf_{R}\dis(R),\]
where $R$ ranges over all correspondences between $X$ and $Y$.
It was also observed in~\cite{kalton1999distances} that
\begin{equation}
\label{eq:dgh}
2\cdot d_{\gh}(X,Y) = \inf_{g,h}\max\{\dis(g),\dis(h),\codis(g,h)\},\end{equation}
where $g\colon X \to Y$ and $h\colon Y \to X$ are any functions.
It follows that $2\cdot d_{\gh}(X,Y)$ is at least as large as the infimum, over all functions $g\colon X\to Y$, of the distortion of $g$.
Interestingly, our best known lower bounds on the Gromov--Hausdorff distance between spheres only rely on lower bounding the distortion (not the codistortion).

\subsection{Background on Borsuk--Ulam theorems}
\label{ssec:background-bu}

The Borsuk--Ulam theorem is a result from algebraic topology with wide-ranging applications:

\begin{theorem}[Borsuk~\cite{borsuk1933drei}]
\label{thm:borsuk-ulam}
For any map $f \colon S^{n} \to \R^{n}$, there exists $x \in S^{n}$ with $f(x) = f(-x)$.
\end{theorem}

We give two equivalent formulations; we leave the equivalence as a simple exercise:

\begin{theorem}
\label{thm:borsuk-ulam-alternate}
Any odd map $g \colon S^{n} \to \R^{n}$ has a zero.
\end{theorem}

\begin{theorem}
\label{thm:borsuk-ulam-odd}
There does not exist an odd map $h \colon S^{n} \to S^{n - 1}$.
\end{theorem}

For $n = 0, 1$, these statements either are trivial or are simple consequences of the intermediate value theorem.
For larger $n$, proofs typically use machinery from algebraic topology (for example, the \emph{degree} or the \emph{Lefschetz number} of a map), though more elementary proofs are also available.
For outlines of several styles of proofs of the Borsuk--Ulam theorem, see~\cite{steinlein1993spheres,matousek2003using}.

The Borsuk--Ulam theorem is foundational to the field of topological combinatorics, as exemplified by Lov\'{s}sz's 1978 proof~\cite{Lovasz1978} of Kneser's conjecture about the chromatic number of Kneser graphs.
The Borsuk--Ulam theorem finds applications across various mathematical disciplines, for example in functional analysis (e.g., to prove the Hobby--Rice theorem~\cite{hobby1965moment}), in differential equations (e.g., to prove that there are infinitely many solutions for a system of nonlinear elliptic partial differential equations~\cite{michalek1989zp}), and in mathematical economics (e.g., to prove the existence of equilibrium with incomplete markets~\cite{husseini1990existence}).

We now introduce some basic notions from equivariant topology.
All of the below is specialized to $\Z/2$, the cyclic group of order two, but also evidently generalizes to other groups.

\begin{itemize}
\item
A \emph{$\Z/2$ space} is a topological space $X$ equipped with an \emph{involution map}, denoted by $x \mapsto -x$, such that $-(-x) = x$ for all $x \in X$.
We say a $\Z/2$ space $X$ is \emph{free} if $-x \ne x$ for all $x \in X$.
\item Given a subset $X'\subseteq X$ of a $\Z/2$ space, we define $-X'\coloneqq\{-x~|~x\in X'\}$.
Furthermore, we say $X'$ is \emph{centrally-symmetric} (or \emph{$\Z/2$ invariant}) if $X'=-X'$.
\item
If $X$ and $Y$ are $\Z/2$ spaces, then a function $f \colon X \to Y$ is \emph{$\Z/2$ equivariant} (or \emph{odd}) if $f(-x) = -f(x)$ for all $x \in X$.
(Similarly, we may describe a map as being \emph{odd}.)
\item
If $X$ is a $\Z/2$ space, then the identity map on $X$ is an odd map.
\item
If $X, Y, Z$ are $\Z/2$ spaces, and $f \colon Y \to Z$, $g \colon X \to Y$ are odd, then $f \circ g$ is odd.
\item
The sphere $S^{n}$ is a $\Z/2$ space, since it inherits the involution map of $\R^{n + 1}$.
\end{itemize}

We now give one representative application of the Borsuk--Ulam theorem, a topological generalization of Radon's theorem on convex sets~\cite{radon1921mengen}:

\begin{theorem}[Bajm\'{o}czy, B\'{a}r\'{a}ny~\cite{bajmoczy1979common}]
\label{thm:bajmoczy-barany}
Let $\Delta^{n + 1}$ be the $(n + 1)$-dimensional simplex in $\R^{n + 2}$.
Then for any map $f \colon \Delta^{n + 1} \to \R^{n}$, there exist $x, y \in \Delta^{n + 1}$ on disjoint faces with $f(x) = f(y)$.
\end{theorem}

\begin{proof}[Proof sketch]
Assume for contradiction that there are no such $x$ and $y$.
We define two topological spaces from $\Delta^{n + 1}$ and $\R^{n}$, by taking a \emph{deleted product} of each in slightly different ways:
\begin{itemize}
\item Let $(\Delta^{n + 1})^{2}_{\Delta}$ be the space of pairs $(x_{1}, x_{2}) \in (\Delta^{n + 1})^{2}$, such that $x_{1}, x_{2}$ are on disjoint faces.
\item Let $(\R^{n})^{2}_{\Delta}$ be the space of pairs $(y_{1}, y_{2}) \in (\R^{n})^{2}$, such that $y_{1} \ne y_{2}$.
\end{itemize}
Note that both $(\Delta^{n + 1})^{2}_{\Delta}$ and $(\R^{n})^{2}_{\Delta}$ are $\Z/2$ spaces; the involution map in each case swaps the two coordinates.
Then under our assumption, $f$ induces an odd map $f^{2}_{\Delta} \colon (\Delta^{n + 1})^{2}_{\Delta} \to (\R^{n})^{2}_{\Delta}$ given by $(x_{1}, x_{2}) \mapsto (f(x_{1}), f(x_{2}))$.
The verification that $f^{2}_{\Delta}$ is odd goes as follows:
\begin{equation*}
f^{2}_{\Delta}(-(x_{1}, x_{2})) = f^{2}_{\Delta}(x_{2}, x_{1}) = (f(x_{2}), f(x_{1})) = -(f(x_{1}), f(x_{2})) = -f^{2}_{\Delta}(x_{1}, x_{2}).
\end{equation*}
It can be shown that there exist odd maps $S^{n} \to (\Delta^{n + 1})^{2}_{\Delta}$ and $(\R^{n})^{2}_{\Delta} \to S^{n - 1}$.
Then the composite map
\[ S^{n} \to (\Delta^{n + 1})^{2}_{\Delta} \to (\R^{n})^{2}_{\Delta} \to S^{n - 1} \]
is odd, contradicting Borsuk--Ulam (specifically, Theorem~\ref{thm:borsuk-ulam-odd}).
Therefore, there exist points $x, y \in \Delta^{n + 1}$ on disjoint faces, such that $f(x) = f(y)$, as desired.
\end{proof}

The proof above suggests defining the concepts of \emph{index} and \emph{coindex} below, which allow us to use spheres of various dimensions as a measuring stick for the topological complexity of a $\Z/2$ space.
Here we give definitions and a few basic facts; see~\cite[Chapter 5]{matousek2003using} for more background.
\begin{itemize}
\item
The \emph{$\Z/2$ index} (or just \emph{index}) of a $\Z/2$ space $X$ is defined to be
\[\ind(X)\coloneqq \min\{k\geq 0 ~|~ \text{there exists an odd map } X\to S^k\}.\]
\item
The \emph{$\Z/2$ coindex} (or just \emph{coindex}) of a $\Z/2$ space $X$ is defined to be
\[\coind(X)\coloneqq \max\{k\geq 0 ~|~ \text{there exists an odd map } S^k\to X\}.\]
\item
For all $n \ge 0$, we have $\ind(S^{n}) = \coind(S^{n}) = n$, by the Borsuk--Ulam theorem.
\item
For all $\Z/2$ spaces $X$, we have $\coind(X) \le \ind(X)$, by the Borsuk--Ulam theorem.
\item
If there exists an odd map $X \to Y$, then $\ind(X) \le \ind(Y)$ and $\coind(X) \le \coind(Y)$.
\item
If the $\Z/2$ space $X$ is not free, then $\ind(X)=\coind(X) = \infty$ because we may construct an odd map $S^{k} \to X$ for any $k \ge 0$ by taking the constant map to a fixed point of the $\Z/2$ action on $X$.
\end{itemize}
In the proof of Theorem~\ref{thm:bajmoczy-barany}, the existence of an odd map $S^{n} \to (\Delta^{n + 1})^{2}_{\Delta}$ shows that we have $\coind((\Delta^{n + 1})^{2}_{\Delta}) \ge n$, and the existence of an odd map $(\R^{n})^{2}_{\Delta} \to S^{n - 1}$ shows that $\ind((\R^{n})^{2}_{\Delta}) \le n - 1$.
But then, using the existence of the odd map $f^{2}_{\Delta} \colon (\Delta^{n + 1})^{2}_{\Delta} \to (\R^{n})^{2}_{\Delta}$, we have
\begin{equation*}
n \le \coind((\Delta^{n + 1})^{2}_{\Delta}) \le \coind((\R^{n})^{2}_{\Delta}) \le \ind((\R^{n})^{2}_{\Delta}) \le n - 1,
\end{equation*}
a contradiction.
We will use the concepts of index and coindex later in the paper.

We will also make use of the concept of a \emph{$k$-connected} space:
A space $X$ is \emph{$k$-connected} if the homotopy groups $\pi_{i}(X)$ are trivial for all $i \le k$.
For example, $X$ is 0-connected if and only if $X$ is path-connected, and
$X$ is 1-connected if and only if $X$ is simply connected.
If a CW complex $X$ is $k$-connected, and if a CW complex $Y$ is $\ell$-connected, then their join $X * Y$ is $(k + \ell + 2)$-connected.
An important property for us is the following:
\begin{equation}
\label{eq:connected-coindex}
\text{ 
If a $\Z/2$ space $X$ is $(k - 1)$-connected, then $\ind(X) \ge \coind(X) \ge k$.
}
\end{equation}
See Proposition~5.3.2 (iv) of~\cite{matousek2003using}, and its proof, for an explanation of this fact.
The proof proceeds as follows.
Pick any point in $X$, and then reflect under the $\Z/2$ action, to get an odd map $S^0\to X$.
Since $\pi_0(X)$ is trivial, we can connect these two points by a path, and then reflect that path via the $\Z/2$ action to get an odd map $S^1\to X$.
Since $\pi_1(X)$ is trivial, we can fill in this map of the circle with a disk, and then reflect via the $\Z/2$ to get an odd map $S^2\to X$.
We continue inductively in this manner, where at the second-to-last step we have obtained an odd map $S^{k-1}\to X$.
Since $\pi_{k-1}$ is trivial, we can fill in with a $k$-dimensional disk, and reflect to get an odd map $S^k\to X$, as desired.

Finally, we define $\Z/2$ versions of some standard concepts from topology:
\begin{itemize}
\item
A \emph{$\Z/2$ metric space} is a $\Z/2$ space $X$ which is also a metric space, and that satisfies $d_X(x, x') = d_X(-x, -x')$ for all $x, x' \in X$.
\item
Let $X, Y$ be $\Z/2$ spaces, and let $f_{0}, f_{1} \colon X \to Y$ be odd maps.
Then a \emph{$\Z/2$ homotopy} from $f_{0}$ to $f_{1}$ is a map $H \colon X \times [0, 1] \to Y$, such that $H(-, 0) = f_{0}$, $H(-, 1) = f_{1}$, and $H(-, t)$ is odd for all $t \in [0, 1]$.
In this case, we say $f_{0}$ and $f_{1}$ are \emph{$\Z/2$ homotopic}.
\item
Let $X, Y$ be $\Z/2$ spaces.
We say that $X, Y$ are \emph{$\Z/2$ homotopy equivalent}, denoted $X \simeq_{\Z/2} Y$, if there exist odd maps $f \colon X \to Y$ and $g \colon Y \to X$, such that $f \circ g$ is $\Z/2$ homotopic to the identity map on $Y$, and $g \circ f$ is $\Z/2$ homotopic to the identity map on $X$.
\end{itemize}
Note that if $X \simeq_{\Z/2} Y$, then $\ind(X) = \ind(Y)$ and $\coind(X) = \coind(Y)$.

\subsection{Background on Vietoris--Rips complexes}
\label{ssec:background-vr}

\subsection*{Simplicial complexes}

We identify a simplicial complex with its geometric realization.
For example, if $\{x_0,\ldots,x_m\}$ is a simplex in a simplicial complex, then we may write $\sum_{i=0}^m\lambda_i x_i$ to refer to a point in the geometric realization of this simplicial complex, where the barycentric coordinates $\lambda_i\ge 0$ satisfy $\sum_i \lambda_i=1$.
A simplicial map between two simplicial complexes indeed deserves the name ``map,'' since it induces a continuous function between geometric realizations.

\subsection*{Vietoris--Rips complexes}

For $X$ a metric space and $r \ge 0$, the \emph{Vietoris--Rips simplicial complex $\vr{X}{r}$} has vertex set $X$, and a nonempty finite subset $\sigma\subseteq X$ is a simplex when $\diam(\sigma)\le r$.
See Figure~\ref{fig:VR}.

The Vietoris--Rips complex is a clique complex (also called flag complex), which means that for every non-empty finite $\sigma \subseteq X$, the simplex $\sigma$ is in $\vr{X}{r}$ if and only if the edge $\{u,v\}$ is in $\vr{X}{r}$ for every pair $u,v \in \sigma$.
This property makes the Vietoris--Rips complex of a finite space suitable to be encoded in a computer, as the information of the 1-skeleton determines the whole complex.

The Vietoris--Rips complex was defined independently by Leopold Vietoris~\cite{Vietoris27} and Eliyahu Rips and has been studied for different reasons along the years; see~\cite{vietoris-obituary,Hausmann1995} for some history.
If $r$ is large enough, Rips used it to show that every hyperbolic group $G$ acts geometrically (by proper and cocompact isometries) on a contractible space, which is none other than $\vr{G}{r}$.
Here, the group $G$ is equipped with the metric induced by the shortest path distance in the Cayley graph $\Gamma(G,S)$ with respect to some generating set $S$ for $G$.
A key consequence of this result is that hyperbolic groups are finitely presented~\cite[Proposition 17, Chapter 4]{GH1990}.
Although it seems that Rips did not publish the result himself, Gromov attributes it to him in Lemma~1.7.A and Section 2.2 of~\cite{Gromov1987}.

\begin{figure}[h]
\includegraphics[width=0.9in]{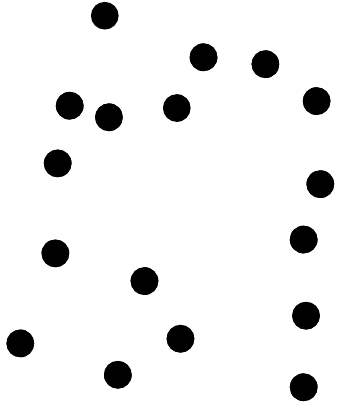}
\hspace{0.1in}
\includegraphics[width=0.9in]{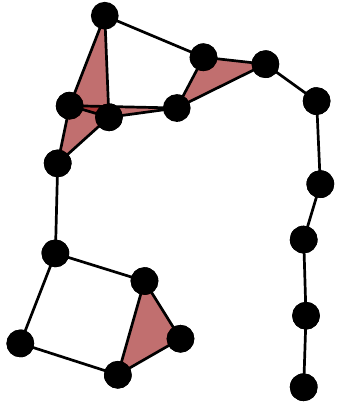}
\hspace{0.1in}
\includegraphics[width=0.9in]{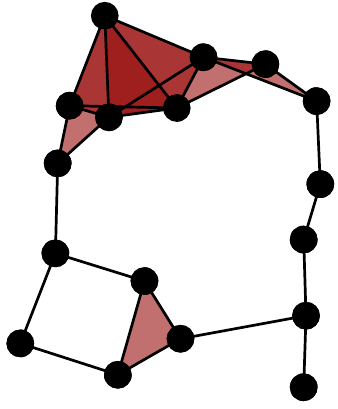}
\hspace{0.1in}
\includegraphics[width=0.9in]{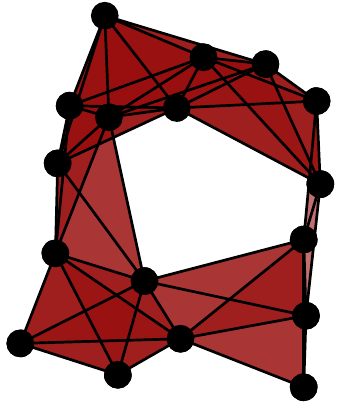}
\hspace{0.1in}
\includegraphics[width=0.9in]{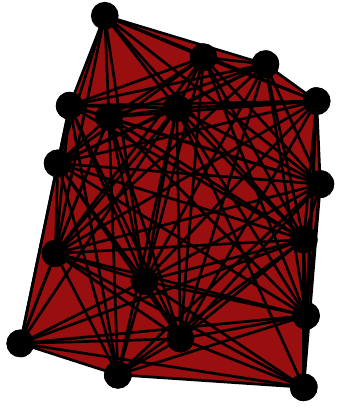}
\captionsetup{width=1\textwidth}
\caption{A metric space $X$ with 17 points, and its Vietoris--Rips complex $\vr{X}{r}$ at four different increasing values of $r>0$.}
\label{fig:VR}
\end{figure}

When $r>0$ is small, on the other hand,
a theorem due to Hausmann~\cite{Hausmann1995} implies that for a given compact Riemannian manifold $M$, there exists some $0<\varepsilon$ such that $M\simeq \vr{M}{r}$ whenever $0<r<\varepsilon$.
Researchers in applied topology are interested in the topology of $\vr{X}{r}$ over all values $r>0$ as a tool to coarsely study the shape of a finite point cloud $X$.
See, for instance, Section 2.3 of~\cite{Carlsson2009}.
These experimental studies of the ``shape of data'' are aided by the fact that the Vietoris--Rips complex is a clique or flag simplicial complex whose persistent homology is relatively efficient to compute~\cite{bauer2021ripser}.
In this paper, we allow the scale parameter $r$ to become large enough so as to change the topology (e.g., the (co)index) of the simplicial complex.

For $X$ a $\Z/2$ metric space and $r\geq 0$, we extend the involution on $X$ to an involution on $\vr{X}{r}$ by defining
\[ \textstyle{
-\left(\sum_i \lambda_i x_i\right)\coloneqq \sum_i\lambda_i (-x_i).
} \]
If $X$ is a free $\Z/2$ metric space, then note that $\vr{X}{r}$ is a free $\Z/2$ space whenever $r<\inf_{x\in X}d_X(x,-x)$.
In particular, $\vr{S^n}{r}$ is a free $\Z/2$ space for $r<\pi$.

A simplicial map between two simplicial complexes induces a continuous map on the geometric realizations of those smplicial complexes.
Therefore, the following lemma shows that Vietoris--Rips complexes are a tool for transforming arbitrary functions between metric spaces into continuous maps between topological spaces; see also~\cite[Lemma~4.3]{ChazalDeSilvaOudot2014}.
Despite the popularity of Vietoris--Rips complexes, this perspective of using Vietoris--Rips complexes to study discontinuous functions appears to be new.

\begin{lemma}\label{lem:ind-map-distortion}
A function $f\colon X \to Y$ between metric spaces induces a simplicial map $\overline{f}:\vr{X}{r}\to\vr{Y}{r+\dis(f)}$ for any $r\ge0$.
If $f$ is an odd function, then $\overline{f}$ is also odd.
\end{lemma}
    
\begin{proof}
Define $\overline{f}\colon \vr{X}{r}\to\vr{Y}{\dis(f)+r}$ by sending a vertex $x\in X$ to $f(x)\in Y$, and then extending linearly to simplices.
In other words,
\[\overline{f}([x_0,\ldots,x_m])=[f(x_0),\ldots,f(x_m)].\]
Observe that if $\diam(\sigma) \le r$ then, by the definition of distortion, $\diam(f(\sigma)) \le r+\dis(f)$.
Thus, $\overline{f}$ is well-defined, simplicial, and continuous (on the underlying geometric realizations), i.e.\ it is a map.
    
If both $X$ and $Y$ are $\Z/2$ metric spaces and $f$ is an odd function, then we see that $\overline{f}$ is an odd map:
\begin{align}
\label{eq:f-odd}
\textstyle{\overline{f}\left(-\sum_i\lambda_i x_i\right) = \overline{f}\left(\sum_i\lambda_i (-x_i)\right)} &= \textstyle{\sum_i\lambda_i f(-x_i)} \\
&=\textstyle{\sum_i\lambda_i (-f(x_i))=-\sum_i\lambda_i f(x_i).} \nonumber
\end{align}
\end{proof}

Lemma~\ref{lem:ind-map-distortion} shows how to turn a possibly discontinuous function into a continuous one; a precursor of this idea is present in~\cite{dubins1981equidiscontinuity}.
In a similar spirit,~\cite{berestovskii2007uniform,brodskiy2013rips,plaut2013discrete,rieser2021cech} study the induced maps on the fundamental group of the Vietoris--Rips complexes of metric spaces via the discrete homotopy approach: one allows paths and homotopies to have discontinuities of size $\varepsilon$.
By considering maps between metric spaces that induce maps between Vietoris--Rips complexes up to a finite amount of shift, Cencelj et al.~\cite{cencelj2012combinatorial} studied the coarse geometry or large-scale properties of metric spaces.

\subsection*{Vietoris--Rips metric thickenings}
Let $X$ be a metric space and let $r\geq 0$.
The \emph{Vietoris--Rips metric thickening} $\vrm{X}{r}$ of $X$ at scale $r$ is the set of probability measures~$\mu$ in $X$ whose support $\supp(\mu)$ is finite and has diameter at most~$r$, equipped with the $1$-Wasserstein metric of optimal transport~\cite{AAF}.
The superscript $\mathrm{m}$ denotes ``metric'', since the metric thickening $\vrm{X}{r}$ is a metric space, whereas the simplicial complex $\vr{X}{r}$ may not be metrizable if $X$ is not discrete.
By identifying each point $x_i \in X$ with the Dirac measure~$\delta_{x_i}$, we can write elements $\mu \in \vrm{X}{r}$ as convex combinations $\mu = \sum_{i=0}^m \lambda_i\delta_{x_i}$, where $\lambda_i \ge 0$, $\sum_i \lambda_i = 1$, and $x_0, \dots, x_m \in X$ with $d_X(x_i,x_j) \le r$ for all $0\le i,j\le m$.
In this way there is a natural isometric embedding from $X$ into~$\vrm{X}{r}$, via the injective map $x\mapsto \delta_x$.
Furthermore, note that the underlying set of the metric thickening $\vrm{X}{r}$ is equal to the underlying set of (the geometric realization of) the simplicial complex $\vr{X}{r}$, although the topology of these two spaces may differ~\cite{AAF}.
In analogy with Hausmann's theorem for simplicial complexes~\cite{Hausmann1995}, metric thickenings are known to recover the homotopy type of the underlying metric space in certain situations~\cite{AAF,AM}.

Occasionally, it will be convenient to work with the metric thickenings instead of simplicial complexes.
However, we will not emphasize metric thickenings and instead refer the reader to~\cite{ABF,HA-FF-ZV,AdamsHeimPeterson,AMMW,gillespie2024vietoris} for further work on these spaces.

\section{Proof of the Main Theorem}
\label{sec:main-theorem-proof}

We are prepared to prove our Main Theorem, which lower bounds the distortion of odd maps between spheres, and hence also the Gromov--Hausdorff distance between spheres of different dimensions.
We will make use of Vietoris--Rips complexes in order to transform an odd function $f$ between spheres into a continuous odd map between Vietoris--Rips complexes of spheres, where the allowable choices of scale parameters depend on how much the function $f$ distorts distances.
Towards these ends, we consider the coindex of Vietoris--Rips complexes of spheres.
We recall the definition of $c_{n,k}$ from Section~\ref{sec:intro}.

\begin{definition-cnk}
For $k\ge n$, we define
\[c_{n,k} \coloneqq \inf\{r\ge 0 \mid \text{there exists an odd map }S^k \to \vr{S^n}{r}\}.\]
\end{definition-cnk}

That is, $c_{n,k}$ is the infimum over all $r\ge 0$ for which $k\le \coind(\vr{S^n}{r})$.
We think of $c_{n,k}$ as the amount we need to ``thicken'' $S^n$ until it admits an odd map from $S^k$.

Let $k\ge n$, and let $f\colon S^k \to S^n$ be an odd function.
In our Main Theorem, we prove $\dis(f)\geq c_{n, k}$.
We remark that Proposition~5.2 of~\cite{lim2023gromov} proves that the distortion of a function is lower bounded by its modulus of discontinuity, which in turn can be controlled as in~\cite{dubins1981equidiscontinuity}.
In Section~\ref{sec:gen-ds} we show that our lower bound on distortion can be strengthened into an analogous lower bound on the modulus of discontinuity of odd functions $S^k \to S^n$.

Our proof of the Main Theorem relies on the following lemma.
We say a subset $A$ of a metric space $X$ is an \emph{$\varepsilon$-covering} if for every point $x\in X$, there exists a point $a\in A$ with $d_X(a,x)<\varepsilon$, i.e.\ with $x \in B(a,\varepsilon)$.

\begin{lemma}\label{lem:covering}
For $X \subset S^k$ a finite $\tfrac{\varepsilon}{2}$-covering with $X=-X$ (that is, $X$ is centrally-symmetric), there exists an odd map
$\phi \colon S^k \to \vr{X}{\varepsilon}$.
\end{lemma}

\begin{proof}
We use the following ``partition of unity'' idea from the proof of stability in~\cite{AMMW,MoyMasters}.
It suffices to consider $\varepsilon < \pi$, since otherwise $\vr{X}{\varepsilon}$ is not a free $\Z/2$ space and $\coind(\vr{X}{\varepsilon})=\infty$.
Let $\{\rho_x\}_{x\in X}$ be a $\Z/2$ invariant partition of unity subordinate to the cover $\{B\left(x, \tfrac{\varepsilon}{2}\right)\}_{x\in X}$ of $S^k$.
That is,
\begin{itemize}
\item $\rho_x$ is a nonnegative continuous real-valued function supported in $B\left(x, \tfrac{\varepsilon}{2}\right)$ for each $x\in X$,
\item $\sum_{x\in X} \rho_x(y) = 1$ for all $y\in S^k$, and
\item $\rho_{-x}(-y) = \rho_{x}(y)$ for all $x\in X$ and $y\in S^k$.
\end{itemize}
To see that such a $\Z/2$ invariant partition of unity exists, note that it can be obtained from a (standard) partition of unity on the quotient space $\RP^n$.

Define the map $\phi\colon S^k \to \vr{X}{\varepsilon}$ by
$\phi(y) \coloneqq \sum_{x\in X} \rho_x(y)\ x$.
Note that any point $x$ whose coefficient in $\phi(y)$ is positive must have $d_{S^k}(x, y)< \tfrac{\varepsilon}{2}$ because $\rho_x$ is supported on $B\left(x, \tfrac{\varepsilon}{2}\right)$.
Therefore, $\diam(\{x\in X \mid \rho_x(y)>0\}) < \varepsilon$, so $\phi(y)$ is a well-defined point in $\vr{X}{\varepsilon}$.
Note that $\phi$ is continuous since each $\rho_x$ is.
Lastly,
\[ \phi(-y)
= \sum_{x\in X} \rho_x(-y)\ x
= \sum_{x\in X} \rho_{-x}(y)\ x
= \sum_{x\in -X} \rho_{x}(y)\ (-x),
\]
which, after applying $X=-X$, is equal to \[\sum_{x\in X} \rho_{x}(y)\ (-x)
= \sum_{x\in X} \rho_{x}(-y)\ x
=-\phi(y).\]
Thus, $\phi$ is an odd map.
\end{proof}

We remark that choosing a different partition of unity will produce a map that is homotopic to $\phi$.
Indeed, given two partitions of unity $\{\rho^1_x\}_{x\in X}$ and $\{\rho^2_x\}_{x\in X}$, the homotopy between the corresponding maps $\phi_1(y) \coloneqq \sum_{x\in X} \rho^1_x(y)\ x$ and $\phi_2(y) \coloneqq \sum_{x\in X} \rho^2_x(y)\ x$ can be given by a straight line homotopy $H(-,t) \coloneqq t\phi_1+(1-t)\phi_2$.

\begin{figure}[h]
\centering
\begin{tikzcd} [row sep=-2pt,column sep=0pt]
S^k 
\ar[r,"{\substack{\text{partition} \\ \text{of unity}}}"]
&
\vr{X}{\varepsilon}
\ar[r,""]
&
\vr{S^n}{\dis(f)+\varepsilon}
\\
&
 {[x_0,\ldots,x_m]}
\ar[r, maps to]
&
{[f(x_0),\ldots,f(x_m)]}
\\
\hspace{2em}
\begin{tikzpicture}[scale=0.6]
\filldraw[fill=none](0,0) circle (1.5);
\draw [dashed](-1.5,0) to[out=90,in=90,looseness=.5] (1.5,0);
\draw (1.5,0) to [out=270,in=270,looseness=.5] (-1.5,0);
\node[right] at (1,-1)  {$S^k$};
\end{tikzpicture}
&
\hspace{2em}
\begin{tikzpicture}[scale=0.6]
\filldraw[fill=none,color=gray](0,0) circle (1.5);
\draw [dashed, color=gray](-1.5,0) to[out=90,in=90,looseness=.5] (1.5,0);
\draw [color=gray] (1.5,0) to [out=270,in=270,looseness=.5] (-1.5,0);
\node[right,color=gray] at (1,-1)  {$S^k$};
\node[fill=red,circle,inner sep=1pt] (a) at (.8,1.1) {};
\node[fill=red,circle,inner sep=1pt] (b) at (1.3,-.2) {};
\node[fill=red,circle,inner sep=1pt] (c) at (.3,.4) {};
\draw [color=red] (a)--(b)--(c)--(a);
\node[draw=red, fill=none, circle,inner sep=.9pt] (-a) at (-.8,-1.1) {};
\node[draw=red, fill=none,circle,inner sep=.9pt] (-b) at (-1.3,.2) {};
\node[draw=red, fill=none,circle,inner sep=.9pt] (-c) at (-.3,-.4) {};
\draw [color=red] (-a)--(-b)--(-c)--(-a);
\end{tikzpicture}
&
\hspace{2em}
\begin{tikzpicture}[scale=0.6]
\filldraw[fill=none,color=black](0,0) circle (1.5);
\node[right,color=black] at (1,-1)  {$S^n$};
\node[fill=red,circle,inner sep=1pt] (a) at (1.2,0.9) {};
\node[fill=red,circle,inner sep=1pt] (b) at (1.5,-.2) {};
\node[fill=red,circle,inner sep=1pt] (c) at (.5,1.4) {};
\draw [color=red] (a)--(b)--(c)--(a);
\node[draw=red, fill=none, circle,inner sep=.9pt] (-a) at (-1.2,-.9) {};
\node[draw=red, fill=none,circle,inner sep=.9pt] (-b) at (-1.5,.2) {};
\node[draw=red, fill=none,circle,inner sep=.9pt] (-c) at (-.5,-1.4) {};
\draw [color=red] (-a)--(-b)--(-c)--(-a);
\end{tikzpicture}
\end{tikzcd}
\caption{Proof, in our Main Theorem, that odd functions $f\colon S^k \to S^n$ for $k\ge n$ have distortion at least~$c_{n,k}$.}
\label{fig:proof-sketch}
\end{figure}
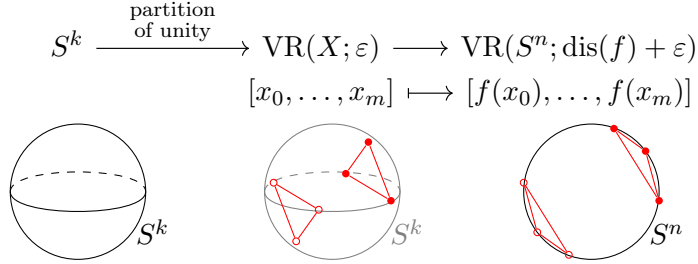

We are now ready to prove our Main Theorem.

\begin{theorem-main}
For all $k \ge n$, the following inequalities hold:
\begin{align*}
2\cdot d_{\gh}(S^n, S^k)
&\ge \inf \left\{ \dis(f) \mid f \colon S^k \to S^n \textnormal{ is odd}\right\} \\
&\ge \inf \left\{r\ge 0 \mid \exists \textnormal{ odd }S^k \to \vr{S^n}{r}\right\} \eqqcolon c_{n,k}.
\end{align*}
\end{theorem-main}

\begin{proof}
Let $k\ge n$, and let $f \colon S^k \to S^n$ be an odd function.
We must show that $\dis(f) \ge c_{n,k}$.
Let $\varepsilon > 0$.
Choose a finite $\Z/2$ invariant $\tfrac{\varepsilon}{2}$-covering $X\subset S^k$.
By Lemma~\ref{lem:covering} we get an odd map $S^k \to \vr{X}{\varepsilon}$, and by Lemma~\ref{lem:ind-map-distortion} the restriction map $f|_X \colon X \to S^n$ induces a continuous odd map $\vr{X}{\varepsilon} \to \vr{S^n}{\dis(f)+\varepsilon}$.
Their composition
\[S^k \to \vr{X}{\varepsilon} \to \vr{S^n}{\dis(f)+\varepsilon}\]
is continuous and odd, showing that $\dis(f)+\varepsilon \ge c_{n,k}$ for all $\varepsilon > 0$.
Hence $\dis(f) \ge c_{n,k}$.

The first inequality in the Main Theorem is the helmet trick from~\cite{lim2023gromov}, which for the sake of completeness we briefly explain here.
Lemma~5.5 from~\cite{lim2023gromov} states that any function $h \colon S^k \to S^n$ can be modified to obtain an \emph{odd} function $f\colon S^k \to S^n$ with $\dis(h)\ge \dis(f)$.
Therefore
\begin{align*}
2\cdot d_{\gh}(S^n,S^k) &= \inf_{\substack{g\colon S^n \to S^k \\ h\colon S^k \to S^n}}\max\{\dis(g),\dis(h),\codis(g,h)\} && \text{by~\eqref{eq:dgh}}\\
&\ge \inf \left\{ \dis(h) \mid h \colon S^k \to S^n \right\} \\
&\ge \inf \left\{ \dis(f) \mid f \colon S^k \to S^n \textnormal{ is odd}\right\}.
\end{align*}
\end{proof}

The quantitative power of our Main Theorem will come from Section~\ref{sec:known-values-cnk}, where we explain how to recover the known values of $c_{n,k}$.
For example, we will see that $c_{n,n+1}=r_n$ (Theorem~\ref{thm:c-n-n+1}), and thus our Main Theorem indeed recovers~\cite[Theorem~B]{lim2023gromov} when $k=n+1$.
We will see $c_{1,2\ell}=c_{1,2\ell+1}=\tfrac{2\pi \ell}{2\ell+1}$ (Theorem~\ref{thm:c-1-k}), and therefore $2\cdot d_{\gh}(S^1,S^k) \ge \tfrac{2\pi \ell}{2\ell+1}$ for $k=2\ell,2\ell+1$.
Furthermore, we will see that for all $k\ge n$, $c_{n,k}$ can be bounded from below in terms of the covering number of $k$ points in the projective space $\R P^n$ (Theorem~\ref{thm:proj-packings}).
The combination of these theorems implies that the bound $2\cdot d_{\gh}(S^n, S^k) \ge c_{n,k}$ in our Main Theorem either recovers or improves upon the best known lower bounds on $d_{\gh}(S^n, S^k)$ from~\cite{lim2023gromov}.
In other words, the bound $2\cdot d_{\gh}(S^n, S^k) \ge c_{n,k}$ is potentially tight; see Remark~\ref{rem:improvement}.
Therefore, our Main Theorem shows that a powerful technique for studying the Gromov--Hausdorff distance between spheres is to obstruct the existence of equivariant maps to Vietoris--Rips complexes of spheres.
And, in the opposite direction, further knowledge about Gromov--Hausdorff distances between spheres will place new constraints on the topology of Vietoris--Rips complexes of spheres.

\begin{remark}
\label{rem:odd-distortion-bound}
The same proof technique of our Main Theorem shows that for any $\Z/2$ space $Y$, odd maps $S^k \to Y$ have distortion at least $\inf\{r\ge 0 \mid \exists\text{ an odd map }S^k \to \vr{Y}{r}\}$.
\end{remark}

In analogy with~\cite[Theorem D]{lim2023gromov}, the proof technique of our Main Theorem 
can provide a more general statement.
Let $H_\ge(S^k)$ denote the closed upper hemisphere of the sphere, namely $H_\ge(S^k)\coloneqq \{(x_1,\ldots,x_{k+1})\in S^k~|~x_{k+1}\ge 0\}$.

\begin{theorem}
Let $X$ and $Y$ be bounded metric spaces such that $X$ isometrically embeds into $S^n$ and $Y$ admits an isometric embedding of $H_\ge(S^k)$, for $k \ge n$.
Then
\[2\cdot d_{\gh}(X, Y) \geq \inf\{r\ge 0 \mid \exists\text{ an odd map }S^k \to \vr{X}{r}\} \geq c_{n,k}.\]
\end{theorem}


\section{Known values of $c_{n,k}$}
\label{sec:known-values-cnk}

In this section, we add quantitative power to our Main Theorem by describing the known values of the constants $c_{n,k} \coloneqq \inf \left\{r\ge 0 \mid \exists \text{ odd }S^k \to \vr{S^n}{r}\right\}$.
These results depend on the topology of Vietoris--Rips complexes and thickenings of spheres.
Indeed, the topology of $\vr{S^n}{r}$ constrains how large the scale $r$ must be in order for the complex to admit an odd map from the $k$-sphere.

We begin with some basic properties that follow from the definition of $c_{n,k}$.
The inclusion $S^k \hookrightarrow S^{k'}$ shows that $c_{n,k}\le c_{n,k'}$ for $k\le k'$.
Also, the inclusion $\vr{S^{n'}}{r} \hookrightarrow \vr{S^n}{r}$ shows that $c_{n,k}\le c_{n',k'}$ for $n\ge n'$ and $k\le k'$.
Since $\pi$ is the diameter of $S^n$, it follows that $\vr{S^n}{\pi}$ is contractible, and therefore $c_{n,k}\le \pi$ for all $k\ge n$.

Next, we observe that $c_{n,k}$ has several different equivalent definitions.
The value of $c_{n,k}$ is unchanged if one uses the convention ``$\diam(\sigma)< r$'' (instead of our convention ``$\diam(\sigma)\le r$'') to define which simplices $\sigma$ are in the Vietoris--Rips complex.
Similarly, the value of $c_{n,k}$ is unchanged if one instead uses Vietoris--Rips metric thickenings --- this follows from the $\varepsilon$-interleavings constructed in~\cite{AMMW,MoyMasters}, which in this setting can be made $\Z/2$ equivariant; see also~\cite{gillespie2024vietoris}.

We have $c_{n,n}=0$ since $\vrm{S^n}{0}=S^n$, or alternatively, since $\vr{S^n}{\varepsilon}\simeq_{\Z/2} S^n$ for all $\varepsilon>0$ sufficiently small.
We also have $c_{0,k}=\pi$ for all $k > 0$, which relies on the convention that $\diam(S^0)=\pi$.

\begin{theorem}
\label{thm:c-1-k}
For all $\ell\ge 1$, we have $c_{1,2\ell+1}=c_{1,2\ell}=\tfrac{2\pi \ell}{2\ell+1}$.
\end{theorem}

\begin{proof}
These values are related to the homotopy types of the simplicial complexes $\vr{S^1}{r}$ and of the metric thickenings $\vrm{S^1}{r}$.
The homotopy types of these simplicial complexes are provided in~\cite{AA-VRS1} as $\vr{S^1}{r}\simeq S^{2\ell+1}$ for $\frac{2\pi \ell}{2\ell+1}<r<\frac{2\pi(\ell+1)}{2\ell+3}$; see Figure~\ref{fig:VRSn}.
The homotopy types of these metric thickenings are proven in~\cite{moy2023vietoris} as $\vrm{S^1}{r}\simeq 
S^{2\ell+1}$ for $\frac{2\pi \ell}{2\ell+1}\le r<\frac{2\pi(\ell+1)}{2\ell+3}$.

For $r>\tfrac{2\pi \ell}{2\ell+1}$, $\vr{S^1}{r}$ is $2\ell$-connected.
We apply \eqref{eq:connected-coindex} in order to obtain an odd map $S^{2\ell+1}\to \vr{S^1}{r}$.
This shows that $c_{1,2\ell}\le c_{1,2\ell+1}\le\tfrac{2\pi \ell}{2\ell+1}$.
On the other hand, Section~5.1 of~\cite{ABF} produces an odd map $\vrm{S^1}{r}\to \R^{2\ell}\setminus\{\vec{0}\}\simeq_{\Z/2} S^{2\ell-1}$ for $r<\frac{2\pi \ell}{2\ell+1}$; the same construction also produces an odd map $\vr{S^1}{r}\to \R^{2\ell}\setminus\{\vec{0}\}$.
Therefore, the Borsuk--Ulam theorem implies there cannot exist odd maps $S^{2\ell}\to\vrm{S^1}{r}$ or $S^{2\ell}\to\vr{S^1}{r}$ for $r<\frac{2\pi \ell}{2\ell+1}$.
This shows $c_{1,2\ell+1}\ge c_{1,2\ell}\ge\tfrac{2\pi \ell}{2\ell+1}$.
Hence $c_{1,2\ell+1}=c_{1,2\ell}=\tfrac{2\pi \ell}{2\ell+1}$, as desired.
\end{proof}


\begin{figure}[h]
\centering
\includegraphics[width=\textwidth]{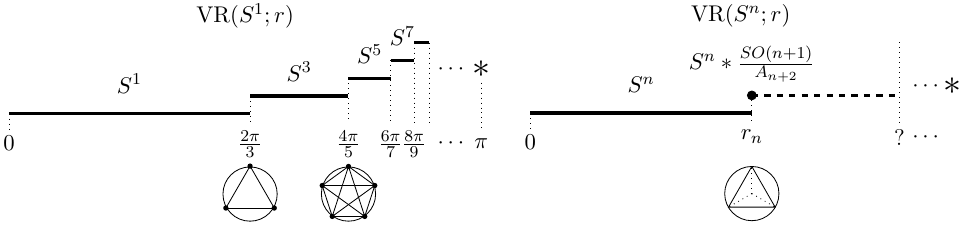}
\caption{The homotopy types of $\vr{S^1}{r}$ and $\vr{S^n}{r}$.}
\label{fig:VRSn}
\end{figure}

\begin{theorem}
\label{thm:c-n-n+1}
For all $n\ge 1$, we have $c_{n,n+2} = c_{n,n+1}=r_n$.
\end{theorem}

\begin{proof}
These values follow from knowledge about the homotopy types of $\vrm{S^n}{r}$ and $\vr{S^n}{r}$.
The related results
\cite[Proposition~5.3]{AAF} and~\cite[Corollary~7.1]{lim2020vietoris} show that $\vrm{S^n}{r}\simeq S^n$ and $\vr{S^n}{r}\simeq S^n$ for all $r < r_n$, respectively.
Furthermore,~\cite[Theorem~5.4]{AAF}\footnote{We refer the reader to~\cite[Section~5.1]{AdamsHeimPeterson} for a gentler introduction to this result.}
provides a homotopy equivalence $\vrm{S^n}{r_n}\simeq 
S^n * \tfrac{\so(n+1)}{A_{n+2}}$; see Figure~\ref{fig:VRSn}.
Since $S^n$ is $(n-1)$-connected and $\tfrac{\so(n+1)}{A_{n+2}}$ is 0-connected, their join $S^n * \tfrac{\so(n+1)}{A_{n+2}}$ is $(n+1)$-connected.
This shows that $c_{n,n+1} \le c_{n,n+2} \le r_n$.
On the other hand,~\cite[Proposition~5.3]{AAF} produces an odd map $\vrm{S^n}{r}\to \R^{n+1}\setminus\{\vec{0}\}\simeq_{\Z/2} S^{n}$ for $r<r_n$; the same construction also produces an odd map $\vr{S^n}{r}\to \R^{n+1}\setminus\{\vec{0}\}$.
Therefore, the Borsuk--Ulam theorem implies there cannot exist odd maps $S^{n+1}\to\vrm{S^n}{r}$ or $S^{n+1}\to\vr{S^n}{r}$ for $r<r_n$.
This shows $c_{n,n+2} \ge c_{n,n+1}\ge r_n$.
Hence $c_{n,n+2}=c_{n,n+1}= r_n$, as desired.
\end{proof}

The exact values of $c_{n,k}$ are not known for $n\ge 2$ and $k\ge n+3$, but we will provide some bounds in the remaining theorem and remarks of this section.

We first provide a bound on $c_{n,k}$ in terms of coverings of projective space.
For $X$ a metric space, let $\cov_X(k)$ be the infimum over all $\varepsilon>0$ such that there exists a finite set $A \subseteq X$ of cardinality $|A| \le k$ such that the balls of radius $\varepsilon$ about $A$ cover $X$, i.e.\ such that $A$ is an $\varepsilon$-covering of $X$.
Let $\RP^n$ be the projective space obtained as the quotient $S^n/(x \sim -x)$, and equipped with the quotient metric.
Explicitly, $d_{\RP^n}(\{x,-x\},\{x',-x'\})=\min(d_{S^n}(x,x'),d_{S^n}(x,-x'))$, so $\RP^n$ has diameter $\tfrac{\pi}{2}$.

Adams, Bush, Frick show in~\cite[Theorem~3]{ABF2} that if $\delta \ge \cov_{\RP^n}(k)$, then there is an odd map $\vrm{S^n}{\pi-2\delta} \to S^{k-1}$, so $\coind(\vrm{S^n}{\pi-2\delta}) \le \ind(\vrm{S^n}{\pi-2\delta}) \le k-1$.
In other words, there is no odd map $S^k \to \vrm{S^n}{\pi-2\delta}$ unless $\delta \le \cov_{\RP^n}(k)$.
After we replace $\pi-2\delta$ with $r$, we see that there is no odd map $S^k \to \vrm{S^n}{r}$ unless $r \ge \pi-2\,\cov_{\RP^n}(k)$.
This is the proof of the following theorem, which follows from~\cite[Theorem~3]{ABF2}, and which is tight when both $n=1$ and $k$ is odd.

\begin{theorem}
\label{thm:proj-packings}
For all $k\ge n\ge 1$, we have $c_{n,k} \ge \pi - 2\,\cov_{\RP^n}(k)$.
\end{theorem}

For any $n\ge 1$, we have $\lim_{k\to\infty}2\,\cov_{\RP^n}(k) = 0$.
Therefore, Theorem~\ref{thm:proj-packings} implies that for any $n\ge 1$, we have $\lim_{k\to\infty} c_{n,k} = \pi$.

\begin{corollary}
\label{cor:odd-distortion-limit}
Fix $n\ge 1$.
The distortion of an odd function $f\colon S^k \to S^n$ tends towards its maximum possible value $\pi$ as $k$ goes to infinity.
\end{corollary}

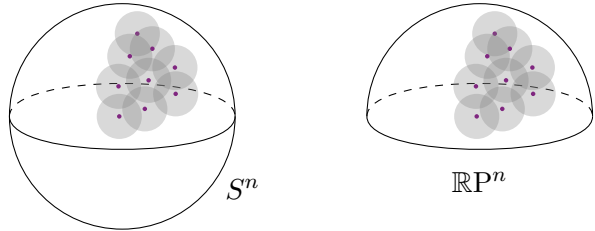
\begin{figure}[htb]
\centering
\begin{tikzcd}
    \begin{tikzpicture}
    \filldraw[fill=none](0,0) circle (1.5);
    \draw [dashed](-1.5,0) to[out=90,in=90,looseness=.5] (1.5,0);
    \draw (1.5,0) to [out=270,in=270,looseness=.5] (-1.5,0);
    \node[right] at (1.2,-1)  {$S^n$};
    \node[fill=violet,circle,inner sep=.6pt] (a) at (.2,1.1) {};
    \fill[gray,opacity=.3](a) circle (.3);
    \node[fill=violet,circle,inner sep=.6pt] (b) at (.4,.9) {};
    \fill[gray,opacity=.3](b) circle (.3);
    \node[fill=violet,circle,inner sep=.6pt] (c) at (.1,.8) {};
    \fill[gray,opacity=.3](c) circle (.3);
    \node[fill=violet,circle,inner sep=.6pt] (d) at (.35,.48) {};
    \fill[gray,opacity=.3](d) circle (.3);
    \node[fill=violet,circle,inner sep=.6pt] (e) at (.7,.65) {};
    \fill[gray,opacity=.3](e) circle (.3);
    \node[fill=violet,circle,inner sep=.6pt] (f) at (.71,.3) {};
    \fill[gray,opacity=.3](f) circle (.3);
    \node[fill=violet,circle,inner sep=.6pt] (g) at (-.05,.4) {};
    \fill[gray,opacity=.3](g) circle (.3);
    \node[fill=violet,circle,inner sep=.6pt] (h) at (.3,.1) {};
    \fill[gray,opacity=.3](h) circle (.3);
    \node[fill=violet,circle,inner sep=.6pt] (i) at (-.04,0) {};
    \fill[gray,opacity=.3](i) circle (.3);
	\end{tikzpicture}
	&
    \begin{tikzpicture}
      \draw (1.5,0) arc(0:180:1.5);
    \draw [dashed](-1.5,0) to[out=90,in=90,looseness=.5] (1.5,0);
    \draw (1.5,0) to [out=270,in=270,looseness=.5] (-1.5,0);
    \node at (0,-1)  {$\RP^n$};
    \node[fill=violet,circle,inner sep=.6pt] (a) at (.2,1.1) {};
    \fill[gray,opacity=.3](a) circle (.3);
    \node[fill=violet,circle,inner sep=.6pt] (b) at (.4,.9) {};
    \fill[gray,opacity=.3](b) circle (.3);
    \node[fill=violet,circle,inner sep=.6pt] (c) at (.1,.8) {};
    \fill[gray,opacity=.3](c) circle (.3);
    \node[fill=violet,circle,inner sep=.6pt] (d) at (.35,.48) {};
    \fill[gray,opacity=.3](d) circle (.3);
    \node[fill=violet,circle,inner sep=.6pt] (e) at (.7,.65) {};
    \fill[gray,opacity=.3](e) circle (.3);
    \node[fill=violet,circle,inner sep=.6pt] (f) at (.71,.3) {};
    \fill[gray,opacity=.3](f) circle (.3);
    \node[fill=violet,circle,inner sep=.6pt] (g) at (-.05,.4) {};
    \fill[gray,opacity=.3](g) circle (.3);
    \node[fill=violet,circle,inner sep=.6pt] (h) at (.3,.1) {};
    \fill[gray,opacity=.3](h) circle (.3);
    \node[fill=violet,circle,inner sep=.6pt] (i) at (-.04,0) {};
    \fill[gray,opacity=.3](i) circle (.3);
	\end{tikzpicture}
	\end{tikzcd}
\caption{(Incomplete) covers of $S^n$ and $\RP^n$.}
\label{fig:cover}
\end{figure}

\begin{remark}
\label{rem:improvement}
Our Main Theorem either recovers or improves upon the best previously known lower bounds on Gromov--Hausdorff distances between spheres, namely~\cite{lim2023gromov}, which proves $2\cdot d_{\gh}(S^n, S^k) \ge \max\{r_n,\pi - 2\,\cov_{S^n}(k+1)\}$ for $k>n$.
Recall that our Main Theorem states that $2\cdot d_{\gh}(S^n, S^k) \geq c_{n,k}$.
To recover the first term $r_n$ from this maximum, use Theorem~\ref{thm:c-n-n+1} and note that $c_{n,k}\ge c_{n,n+1}=r_n$ for $k>n$.
To improve upon the second term $\pi - 2\,\cov_{S^n}(k+1)$ from this maximum, note that $c_{n,k} \ge \pi-2\,\cov_{\RP^n}(k)$ by Theorem~\ref{thm:proj-packings}, and that it is easier to cover the quotient space $\RP^n$ than it is to cover the sphere $S^n$, since distances can only decrease upon taking quotients; see Figure~\ref{fig:cover}.
Furthermore, the $r_n$ and $\pi - 2\,\cov_{S^n}(k+1)$ lower bounds in~\cite{lim2023gromov} are proven using two separate arguments, which are now unified, generalized, and improved upon by our single lower bound $c_{n,k}$.
One specific instance of improvement is $n=1$, when we obtain $2\cdot d_{\gh}(S^1, S^{2\ell}) \ge c_{1,2\ell} = \tfrac{2\pi \ell}{2\ell+1}$ and $2\cdot d_{\gh}(S^1, S^{2\ell+1}) \ge c_{1,2\ell+1} = \tfrac{2\pi \ell}{2\ell+1}$; note $c_{1,k} > r_1$ for $k\ge 4$.
\end{remark}

\begin{remark}
Theorem~2 of~\cite{ABF2} gives an upper bound on the values of $c_{n,k}$ in terms of packings of points in $\RP^n$, and in particular implies that $c_{n,k}<\pi$ for all $k\ge n$.
\end{remark}

\begin{remark}
\label{rem:c-2-7}
The following calculation further illustrates  that the elucidation of the subsequent homotopy types of Vietoris-Rips complexes of spheres can help estimate the numbers $c_{n,k}$.
Partial results for the case  of $S^2$ can be obtained thanks to early work by Katz.
Indeed, by~\cite[
Corollary 7]{lim2020vietoris} and~\cite{katz1989diameter,katz1991neighborhoods}, we know that $\vr{S^2}{r}\simeq S^2 * \tfrac{S^3}{E_6} = S^2 * \tfrac{\so(3)}{A_4}$ for all $r_2 < r < \arccos\left(\tfrac{-1}{\sqrt{5}}\right)$.
Since $S^2 * \tfrac{\so(3)}{A_4}$ is 6-dimensional with a free $\Z/2$ action, 
there is no odd map $S^7 \to S^2 * \tfrac{\so(3)}{A_4}$, and therefore we can conclude that $c_{2,7}\geq \arccos\left(\tfrac{-1}{\sqrt{5}}\right)$.
It is currently open whether the same lower bound holds for $c_{2,6}$ or $c_{2,5}$.
\end{remark}

\begin{remark}
\label{rem:icosahedron}
The 12 vertices of a regular icosahedron inscribed in $S^2$ can be chosen to be $\tfrac{1}{\sqrt{1+\phi^2}}(0,\pm 1,\pm\phi)$, $\tfrac{1}{\sqrt{1+\phi^2}}(\pm 1,\pm\phi,0)$, and $\tfrac{1}{\sqrt{1+\phi^2}}(\pm\phi,0,\pm 1)$, where $\phi\coloneqq\tfrac{\sqrt{5}+1}{2}$ is the golden ratio.
The three vertices $\tfrac{1}{\sqrt{1+\phi^2}}(1,\phi,0)$, $\tfrac{1}{\sqrt{1+\phi^2}}(\phi,0,1)$, and $\tfrac{1}{\sqrt{1+\phi^2}}(\phi,0,-1)$ form a face, and since the geodesic distance between the center of this triangle and one of the vertices is $\arccos\left(\sqrt{\tfrac{5+2\sqrt{5}}{15}}\right)$, we can conclude that $\cov_{S^2}(12)\leq \arccos\left(\sqrt{\tfrac{5+2\sqrt{5}}{15}}\right)$.
Since this set of 12 points in $S^2$ is centrally-symmetric, it produces a set of 6 points in $\RP^2$ showing that $\cov_{\RP^2}(6)\leq \arccos\left(\sqrt{\tfrac{5+2\sqrt{5}}{15}}\right)$.
By our Main Theorem and Theorem~\ref{thm:proj-packings}, we achieve
\[2\cdot d_{\gh}(S^2, S^k)\geq c_{2,k} \ge c_{2,6} \geq\pi-2\,\cov_{\RP^2}(6)\geq\pi-2\,\arccos\left(\sqrt{\tfrac{5+2\sqrt{5}}{15}}\right) \quad \text{for }k\ge 6.\]
This holds for more values of $k$ than~\cite[Proposition~1.11]{lim2023gromov}, which only gives
\[ 2 d_{\gh}(S^2, S^k) \ge \pi - 2\,\cov_{S^2}(k+1) \ge \pi - 2\,\cov_{S^2}(12) \geq\pi-2\,\arccos\left(\sqrt{\tfrac{5+2\sqrt{5}}{15}}\right) \text{for }k\ge 11.\]
See also~\cite[Example~4.5]{ABF2}.
\end{remark}

\begin{remark}
\label{rem:600-cell}
The 600-cell is a convex regular 4-polytope with 600 tetrahedral cells and 120 antipode-preserving vertices in $S^3$.
One can choose those 120 vertices in the following way: 8 vertices obtained from $(0,0,0,\pm 1)$ by permuting coordinates, 16 vertices of the form $\left(\pm\tfrac{1}{2},\pm\tfrac{1}{2},\pm\tfrac{1}{2},\pm\tfrac{1}{2}\right)$, and the remaining 96 vertices are obtained by taking even permutations of $\left(\pm\tfrac{\phi}{2},\pm\tfrac{1}{2},\pm\tfrac{\phi^{-1}}{2},0\right)$, where $\phi\coloneqq\tfrac{\sqrt{5}+1}{2}$ is the golden ratio.
The Euclidean distance between the two closest vertices is $\phi^{-1}$, and hence the geodesic distance between them is $\tfrac{\pi}{5}$.
By direct computation, the four vertices $(1,0,0,0)$, $\left(\tfrac{\phi}{2},\tfrac{1}{2},\tfrac{\phi^{-1}}{2},0\right)$, $\left(\tfrac{\phi}{2},\tfrac{1}{2},-\tfrac{\phi^{-1}}{2},0\right)$, $\left(\tfrac{\phi}{2},\tfrac{\phi^{-1}}{2},0,\tfrac{1}{2}\right)$ form a tetrahedral cell.
Since the geodesic distance between the center of this cell and one of the vertices is $\arccos\left(\tfrac{1+\sqrt{5}}{2\sqrt{3}}\right)$, we can conclude that $\cov_{S^3}(120)\leq \arccos\left(\tfrac{1+\sqrt{5}}{2\sqrt{3}}\right)$.
This implies that $\cov_{\RP^3}(60)\leq \arccos\left(\tfrac{1+\sqrt{5}}{2\sqrt{3}}\right)$.
Finally, from our Main Theorem and Theorem~\ref{thm:proj-packings}, we obtain
\[2\cdot d_{\gh}(S^3, S^{60})\geq c_{3,60}\geq\pi-2\,\cov_{\RP^3}(60)\geq\pi-2\,\arccos\left(\tfrac{1+\sqrt{5}}{2\sqrt{3}}\right).\]
See~\cite[Remark~4.2]{ABF2}.
\end{remark}

\section{A novel upper bound on the Gromov--Hausdorff distance $d_{\gh}(S^n, S^{n+1})$}
\label{sec:super-diag-upper-bound}

We will give a new upper bound on $2\cdot d_\gh(S^n, S^{n+1})$, improving the existing bounds for all $n > 3$.
In particular, we will prove the following theorem.

\begin{theorem-super-diag}
For every $n\ge 1$, we have $2\cdot d_\gh(S^n, S^{n+1}) \le \frac{2\pi}{3}$.
\end{theorem-super-diag}

We first introduce several geometric objects, and recall the current best upper bounds.
For all $n\ge 1$, we may inscribe a regular $(n+1)$-simplex in $S^n$.
Any pair of vertices of the inscribed simplex lie the same geodesic distance apart, and this distance is exactly the quantity
$r_n = \arccos\left(-\frac{1}{n+1}\right)$.
The facets of the inscribed simplex may be projected radially outward, obtaining $(n+2)$ sets that cover $S^n$, and which are additionally closed, geodesically convex, and pairwise isometric.
We call these radially projected facets \emph{regular geodesic simplices} in $S^n$.
Santal\'o~\cite{santalo1946convexregions} computed the diameter of these simplices, which is \[
t_n \coloneqq \begin{cases}
\arccos\left(-\frac{n+1}{n+3}\right)&\text{for $n$ odd},\\
\arccos\left(-\sqrt{\frac{n}{n+4}}\right)& \text{for $n$ even}.
\end{cases}
\]
This diameter is achieved between points at the centers of opposite faces which each contain half the vertices of the simplex (rounded appropriately when there are an odd number of vertices).
Notice that $r_n \le t_n$ for every $n$.
In fact, equality holds only for $n=1$, when $r_1=\tfrac{2\pi}{3}=t_1$.
As $n\to \infty$ we have $r_n\to \tfrac{\pi}{2}$ and $t_n\to \pi$.
In particular, $t_n > \tfrac{2\pi}{3}$ for all $n\ge 2$.

The quantities $r_n$ and $t_n$ played an important role in the work of Lim, Memoli, and Smith~\cite{lim2023gromov}, who showed that $r_n \le 2\cdot d_\gh(S^n, S^{n+1}) \le t_n$.
They also obtained exact results for small $n$, showing that $2\cdot d_\gh(S^1,S^2) = \tfrac{2\pi}{3}= 2\cdot d_\gh(S^1, S^3)$ and $2\cdot d_\gh(S^2,S^3) = r_2$.
Theorem~\ref{thm:super-diag-upper-bound} improves the upper bound on $2\cdot d_\gh(S^n,S^{n+1})$ for all $n>3$, and also improves the upper bound asymptotically---the previous bound converged to $\pi$, while Theorem~\ref{thm:super-diag-upper-bound} bounds it strictly away from $\pi$.

To build towards a proof of Theorem~\ref{thm:super-diag-upper-bound}, we first make a simple observation regarding the distortion of relations which only pair together points that lie a bounded distance from one another.
\begin{lemma}
\label{lem:bounded-stretch}
Let $(X, d_X)$ be a metric space, and let $Y\subseteq X$ be a subspace with induced metric $d_Y$.
Let $R\subseteq X\times Y$ be any relation, and define
\[\varepsilon \coloneqq \sup\{d_X(x,y)\mid (x,y)\in R\}.\]
Then the distortion of $R$ is at most $2\varepsilon$.
\end{lemma}
\begin{proof}
Let $(x,y)$ and $(x',y')$ be in $R$.
We wish to bound $|d_X(x,x') - d_Y(y,y')|$.
Applying the triangle inequality twice, we see that
\[
d_X(x,x') \le d_X(x,y) + d_X(y,y') + d_X(y', x') \le 2\varepsilon + d_X(y,y') = 2\varepsilon + d_Y(y,y').
\]
Hence $d_X(x,x')-d_Y(y,y')\le 2\varepsilon$.
A symmetric application of the triangle inequality shows that $d_Y(y,y') - d_X(x,x') \le 2\varepsilon$.
Together these inequalities imply the desired bound.
\end{proof}

Recall that $H_\ge(S^{n+1})$ denotes the closed upper hemisphere of $S^{n+1}$.
Let $N\in H_\ge(S^{n+1})$ denote the north pole.
We will make use of the map $\tau \colon H_\ge(S^{n+1})\setminus\{N\}\to S^n$ which sends a point in the upper hemisphere to the unique nearest point on the equator.
In other words, for $x\in H_{\ge}(S^{n+1})\setminus \{N\}$, we define $\tau(x)$ to be the result of setting the final coordinate in $x$ to zero, and then normalizing.

In the proof of Theorem~\ref{thm:super-diag-upper-bound} below, we require two important facts.
The most crucial is that to bound $d_\gh(S^n, S^{n+1})$ it suffices to bound the distortion of correspondences between the upper hemisphere $H_\ge(S^{n+1})$ and the equator $S^n$ (see~\cite[Lemma~5.5]{lim2023gromov}).
Second, if $x\neq N$ and $x'$ are points in $H_{\ge}(S^{n+1})$ and $d_{S^{n+1}}(x,x')\ge\tfrac{\pi}{2}$, then $d_{S^{n+1}}(\tau(x), x')\ge d_{S^{n+1}}(x,x')$.
Indeed, $d_{S^{n+1}}(x,x') \ge \tfrac{\pi}{2}$ if and only if $\langle x,x'\rangle \le 0$, and since both $x$ and $x'$ have nonnegative last coordinate we see that $\langle \tau(x), x'\rangle \le \langle x,x'\rangle$, which implies that $d_{S^{n+1}}(\tau(x), x') \ge d_{S^{n+1}}(x,x')$.

\begin{proof}[Proof of Theorem~\ref{thm:super-diag-upper-bound}]
We first construct a correspondence between $S^n$ and $S^{n+1}$, and then we bound its distortion.

\begin{figure}[h]
\centering
\includegraphics[width=2.5in]{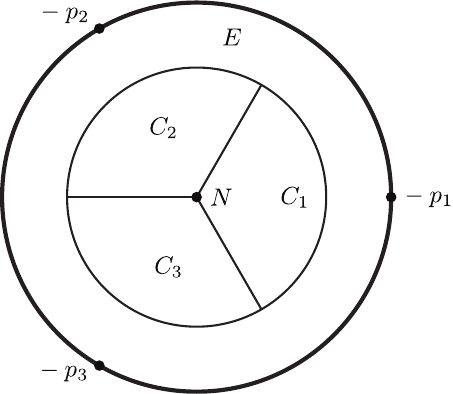}
\caption{The decomposition of $H_\ge(S^{n+1})$ used in the proof of Theorem~\ref{thm:super-diag-upper-bound}.}
\label{fig:hemisphere_decomposition}
\end{figure}

\vspace{1mm}\noindent\textbf{Constructing the correspondence.}
Let $P = \{p_1, p_2, \ldots, p_{n+2}\}$ be the vertices of an inscribed regular $(n+1)$-simplex in $S^n$.
For each $i\in[n+2]\coloneqq\{1,2,\ldots,n+2\}$, let $F_i$ be the geodesic convex hull of $P\setminus \{p_i\}$.
So, for $i\in[n+2]$ the set $F_i$ is a regular geodesic simplex in $S^n$, and its barycenter is $-p_i$.
Define
\[
E \coloneqq \{p\in H_{\ge}(S^{n+1}) \mid d_{S^{n+1}}(p, N) > \tfrac{\pi}{3}\}.
\]
Further, for $i\in[n+2]$ define
\[
C_i \coloneqq \{p\in H_{\ge}(S^{n+1}) \mid p\neq N,\ \tau(p) \in F_i, \text{ and } d_{S^{n+1}}(p,N) \le \tfrac{\pi}{3}\} \cup \{N\}.
\]
So, $E$ is a ``thickened equator'' consisting of the points with distance less than $\tfrac{\pi}{6}$ to the equator, and the various $C_i$ are cones with apex $N$ over the various $F_i$, restricted to a closed ball of radius $\tfrac{\pi}{3}$ around $N$; see Figure~\ref{fig:hemisphere_decomposition}.
Finally, we define a correspondence $R$ between $H_\ge(S^{n+1})$ and $S^n$ as follows:
\[
R \coloneqq  \{(p, \tau(p)) \mid p\in E\} \sqcup \{ (p, -p_i) \mid p\in C_i \text{ for some } i\in[n+2]\}.
\]
Note that this is a correspondence since $E$ and the various $C_i$ cover $H_{\ge}(S^{n+1})$, and since $(p,p)\in R$ for every $p\in S^n$.

\vspace{1mm}\noindent\noindent\textbf{Bounding the distortion.}
We will argue that the distortion of $R$ is at most $\tfrac{2\pi}{3}$.
To this end, let $(x,y)$ and $(x',y')$ be elements of $R$.
To bound $|d_{S^{n+1}}(x,x') - d_{S^{n}}(y,y')|$ we consider the following cases.\smallskip

\noindent \emph{Case 1: Both $x$ and $x'$ lie in $E$.}
By Lemma~\ref{lem:bounded-stretch}, the relation between $E$ and $S^n$ consisting of pairs $(x,\tau(x))$ has distortion at most $\tfrac{\pi}{3}$.
Here we have $y = \tau(x)$ and $y' = \tau(x')$, so $|d_{S^{n+1}}(x,x') - d_{S^{n}}(y,y')|$ is at most $\tfrac{\pi}{3}$.
\smallskip

\noindent \emph{Case 2: Neither $x$ nor $x'$ lie in $E$.}
Here we must have $d_{S^{n+1}}(x, N)\le \tfrac{\pi}{3}$ and $d_{S^{n+1}}(x', N)\le \tfrac{\pi}{3}$.
Hence $d_{S^{n+1}}(x,x')\le \tfrac{2\pi}{3}$.
Moreover, $y$ and $y'$ both lie in $P$, so $d_{S^{n}}(y,y')\le r_n \le \tfrac{2\pi}{3}$.
Thus we have $|d_{S^{n+1}}(x,x') - d_{S^{n}}(y,y')| \le \max\{d_{S^{n+1}}(x,x'), d_{S^{n}}(y,y')\} \le \tfrac{2\pi}{3}$.
\smallskip

\noindent \emph{Case 3: $x\in E, x'\notin E$, and $d_{S^{n+1}}(x,x')\le \tfrac{\pi}{2}$.}
Since $d_{S^{n+1}}(x,x')\le \tfrac{\pi}{2}$, it will suffice to show that $d_{S^{n}}(y,y')-d_{S^{n+1}}(x,x')\le \tfrac{2\pi}{3}$.
Observe that $y = \tau(x)$, so $d_{S^{n+1}}(x,y)\le \tfrac{\pi}{6}$.
Moreover, for some $i\in[n+2]$ we have $x'\in C_i$ and $y'=-p_i$.
Every point in $C_i$ has nonnegative inner product with $-p_i$, so $d_{S^{n+1}}(x',y')\le \tfrac{\pi}{2}$.
Applying the triangle inequality twice, we obtain 
\begin{align*}
    d_{S^{n}}(y,y') &\le d_{S^{n+1}}(y,x) + d_{S^{n+1}}(x,x') + d_{S^{n+1}}(x',y')\\
    &\le \tfrac{\pi}{6} + d_{S^{n+1}}(x,x') + \tfrac{\pi}{2} .
\end{align*}
Hence $d_{S^{n}}(y,y')-d_{S^{n+1}}(x,x')\le \tfrac{2\pi}{3}$ as desired.\smallskip

\noindent \emph{Case 4: $x\in E, x'\notin E$, and $d_{S^{n+1}}(x,x') > \tfrac{\pi}{2}$.}
Since $d_{S^{n+1}}(x,x') > \tfrac{\pi}{2}$, it will suffice to show that $d_{S^{n+1}}(x,x') - d_{S^{n}}(y,y') \le \tfrac{2\pi}{3}$.
We have $y= \tau(x)$, and since $d_{S^{n+1}}(x,x')> \tfrac{\pi}{2}$ this implies that $d_{S^{n+1}}(x,x') \le d_{S^{n+1}}(y,x') \le d_{S^{n+1}}(y,y') + d_{S^{n+1}}(x',y')$.
Consequently, $d_{S^{n+1}}(x,x') - d_{S^{n}}(y,y') \le d_{S^{n+1}}(x',y')$.
Our analysis in the previous case showed that $d_{S^{n+1}}(x',y')\le \tfrac{\pi}{2}$, so the result follows.
\end{proof}

\section{Generalized Dubins--Schwarz inequality}
\label{sec:gen-ds}

The Borsuk--Ulam theorem states that an odd function $S^{n+1} \to S^n$ is discontinuous, but how discontinuous must it be?
One possible quantitative answer is in terms of the \emph{modulus of discontinuity} of a function, which is positive if and only if the function is discontinuous.
Initial results in this direction are given by Dubins and Schwarz in~\cite[Corollary~3]{dubins1981equidiscontinuity}:
The modulus of discontinuity of an odd function $S^{n+1}\to S^n$ is bounded from below by $r_n$, where $r_n\coloneqq \arccos\left(\tfrac{-1}{n+1}\right)$ is the (geodesic) distance between two vertices of the regular $(n+1)$-simplex inscribed in $S^n$.
More generally, for any $k>n$ the modulus of discontinuity and distortion of an odd function $S^k \to S^n$ is at least $r_n$; this follows from the prior facts after pre-composing the odd function $S^k \to S^n$ with an inclusion $S^{n+1}\hookrightarrow S^k$.
However, this lower bound $r_n$ does not depend on $k$.
In this section, we provide improved lower bounds on the modulus of discontinuity of an odd function $S^k \to S^n$, which are weakly increasing and not constant as $k$ increases.

Let $X$ be a topological space, let $Y$ be a metric space, and let $f\colon X\to Y$ be a function.
Then as defined in~\cite{dubins1981equidiscontinuity}, the \emph{modulus of discontinuity of $f$} is
\[
\delta(f)\coloneqq \inf\{\delta\geq 0 \mid \forall x\in X,\ \exists~\text{an open neighborhood } U_x \text{ of } x \text{ s.t.\ } \diam(f(U_x))\leq \delta\}.
\]
Note that $f$ is discontinuous if and only if $\delta(f)>0$.
Restating a result from Dubins and Schwarz~\cite{dubins1981equidiscontinuity} with the geodesic metric instead of the Euclidean metric, we obtain the following:

\begin{theorem}[Dubins--Schwarz inequality; Corollary~3 and Scholium~1 of~\cite{dubins1981equidiscontinuity}]
\label{thm:ds-inequality}
Any odd function $f\colon S^{n+1} \to S^n$ has modulus of discontinuity $\delta(f)\geq r_n$, and this bound is attained.
\end{theorem}

\noindent Thus, we recover not only the Borsuk--Ulam theorem stating that $f$ is discontinuous, but furthermore we obtain a quantitative bound on how discontinuous $f$ must be.

In this section, we generalize the Dubins--Schwarz inequality for odd functions $S^k \to S^n$ for $k\ge n$.
Our primary theorem in this section is the following.

\begin{theorem-odd-modulus-discontinuity-bound}[Generalized Dubins--Schwarz inequality]
Any odd function $f\colon S^k \to S^n$ with $k \ge n$ has modulus of discontinuity $\delta(f)\geq c_{n,k}$.
\end{theorem-odd-modulus-discontinuity-bound}

This theorem generalizes~\cite[Corollary~3]{dubins1981equidiscontinuity} since $c_{n,n+1}=r_n$.
Since Theorem~\ref{thm:proj-packings} implies $\lim_{k\to\infty} c_{n,k} = \pi$, we obtain the following corollary.

\begin{corollary}
\label{cor:odd-modulus-discontinuity-limit}
Fix $n\ge 1$.
The modulus of discontinuity of an odd function $f\colon S^k \to S^n$ tends towards its maximum possible value $\pi$ as $k$ goes to infinity.
\end{corollary}

\begin{remark}
The modulus of discontinuity always lower bounds the distortion by ~\cite[Proposition~5.2]{lim2023gromov}.
Therefore Theorem~\ref{thm:odd-modulus-discontinuity-bound} implies the bound 
$\dis(f) \ge c_{n,k}$ for odd functions $f \colon S^k \to S^n$
from our Main Theorem.
Nevertheless, for ease of exposition, we have presented our results on distortion first, before moving now to the slightly more complicated case of modulus of discontinuity.
\end{remark}

In Section~\ref{ssec:lower-bound} we prove Theorem~\ref{thm:odd-modulus-discontinuity-bound} giving a lower bound on the modulus of discontinuity, and in Section~\ref{ssec:lower-bound-tight} we show that this lower bound is tight.
In Section~\ref{ssec:lower-bound-generalized} we generalize the domain $S^{k}$ in Theorem~\ref{thm:odd-modulus-discontinuity-bound} to instead be any $\Z/2$ space $X$ with coindex at least $k$.
Lastly, in Section~\ref{ssec:Euclidean} we consider spheres equipped with the Euclidean metric.

\subsection{Lower bound on the modulus of discontinuity}
\label{ssec:lower-bound}

To prove Theorem~\ref{thm:odd-modulus-discontinuity-bound}, we need the following lemma, which is similar to Lemma~\ref{lem:ind-map-distortion} except with distortion replaced by the modulus of discontinuity.

\begin{lemma}
\label{lem:ind-map-modulus-discontinuity}
Let $f\colon X \to Y$ be a function between metric spaces, and let $X$ be compact.
Then for any $\varepsilon>0$, there is a sufficiently small $\alpha_\varepsilon>0$ such that $f$ induces a simplicial map $\overline{f}:\vr{X}{\alpha_\varepsilon}\to \vr{Y}{\delta(f) + \varepsilon}$ defined by $\overline{f}([x_0,...,x_m]) = [f(x_0),...,f(x_m)]$.
If $f$ is odd, then so is $\overline{f}$.
\end{lemma}

\begin{proof}
We first show that for any $\varepsilon>0$, there exists $\alpha_\varepsilon>0$ such that for any $x, x'\in X$ with $d_X(x, x')\leq \alpha_\varepsilon$, we have $d_Y(f(x), f(x'))<\delta(f) + \varepsilon$.
From the definition of modulus of discontinuity, for any $x\in X$, there exists $r_x>0$ such that $\diam(f(B(x; r_x)))< \delta(f) + \varepsilon$.
Consider the open cover $\{B\left(x;\frac{r_x}{2}\right)\}_{x\in X}$ of $X$.
As $X$ is compact, there is a finite sub-cover $\{B\left(x_i; \frac{r_{x_i}}{2}\right)\}_{1\leq i \leq N}$.
Choose $\alpha_\varepsilon$ to satisfy $0<\alpha_\varepsilon<\min\left\{\frac{r_{x_1}}{2},\ldots,\frac{r_{x_N}}{2}\right\}$.
Let $x, x'$ be any pair of points in $X$ such that $d_X(x, x')\leq \alpha_\varepsilon$.
Since $\{B\left(x_i; \frac{r_{x_i}}{2}\right)\}_{1\leq i \leq N}$ is a cover of $X$, there is some $B\left(x_i;\frac{r_{x_i}}{2}\right)$ that contains $x$.
Then, 
\[d_X(x',x_i) \le d_X(x',x)+d_X(x,x_i) < \alpha_\varepsilon + \tfrac{r_{x_i}}{2} < r_{x_i}.\]
That is, both $x$ and $x'$ are inside $B(x_i; r_{x_i})$, and therefore \[d_Y(f(x), f(x')) \le \diam(f(B(x; r_x))) <\delta(f) + \varepsilon.\]

Define $\overline{f}:\vr{X}{\alpha_\varepsilon}\to \vr{Y}{\delta(f) + \varepsilon}$ by sending a vertex $x\in X$ to $f(x)\in Y$, and then extending linearly.
Equivalently, $\overline{f}([x_0,\ldots,x_m])=[f(x_0),\ldots,f(x_m)]$.
By the above paragraph, $d_Y(f(x), f(x'))< \delta(f) +\varepsilon$ whenever $d_X(x, x')\leq \alpha_\varepsilon$, and therefore $\overline{f}$ is a well-defined map.
If $f$ is odd, the verification that $\overline{f}$ is also odd is the same as in \eqref{eq:f-odd}.
\end{proof}

We remark that the above lemma is in some sense sharp: if $r<\delta(f)$, then for any $\alpha>0$ the proposed map $\overline{f}:\vr{X}{\alpha}\to \vr{Y}{r}$ defined by sending a vertex $x\in X$ to $f(x)\in Y$ and extending linearly would not be well-defined.

We can now prove Theorem~\ref{thm:odd-modulus-discontinuity-bound}, using a similar structure to the proof of our Main Theorem.

\begin{proof}[Proof of Theorem~\ref{thm:odd-modulus-discontinuity-bound}]
Let $k\ge n$, and let $f\colon S^k \to S^n$ be an odd function.
We must show that $\delta(f)\ge c_{n,k}$.
Let $\varepsilon>0$.
By Lemma~\ref{lem:ind-map-modulus-discontinuity}, there is some $\alpha_\varepsilon>0$ such that $f$ induces a map $\overline{f}\colon\vr{S^k}{\alpha_\varepsilon}\to \vr{S^n}{\delta(f)+ \varepsilon}$.
Choose a finite $\Z/2$ invariant $(\alpha_\varepsilon/2)$-covering $X\subset S^k$.
By Lemma~\ref{lem:covering}, we get an odd map $S^k\to \vr{X}{\alpha_\varepsilon}$.
Restrict the domain of $\overline{f}$ to obtain an odd map $\vr{X}{\alpha_\varepsilon}\to \vr{S^n}{\delta(f)+ \varepsilon}$.
The composition
\[ S^k\to \vr{X}{\alpha_\varepsilon} \to \vr{S^n}{\delta(f)+ \varepsilon}\]
is continuous and odd, showing that $\delta(f)+\varepsilon \ge c_{n,k}$ for all $\varepsilon > 0$.
Hence $\delta(f) \ge c_{n,k}$.
\end{proof}

\begin{remark}
\label{rem:odd-modulus-discontinuity-bound}
The same proof technique shows that any odd map $S^k \to Y$ has modulus of discontinuity at least $\inf\{r\ge 0 \mid \coind(\vr{Y}{r})\ge k\}$.
\end{remark}

\subsection{Tightness of the lower bound on modulus of discontinuity}
\label{ssec:lower-bound-tight}

We now show that Theorem~\ref{thm:odd-modulus-discontinuity-bound} is tight, generalizing the tightness of Theorem~\ref{thm:ds-inequality}.




\begin{theorem}
\label{thm:odd-modulus-discontinuity-bound-tight}
For any $k \ge n$ and $\varepsilon>0$, there exists an odd function $f\colon S^k \to S^n$ with modulus of discontinuity $\delta(f)\le c_{n,k}+\varepsilon$.
\end{theorem}

To prove Theorem~\ref{thm:odd-modulus-discontinuity-bound-tight}, we will need the following lemma.
In this subsection, for the purpose of clarity, we use different notation to differentiate between a simplicial complex $K$ and its geometric realization $|K|$, even though in other sections of this paper we identify simplicial complexes with their geometric realizations.

\begin{lemma}
\label{lem:VR-to-mod}
Let $X$ be a $\Z/2$ topological space, let $Y$ be a $\Z/2$ metric space, and let $r\ge 0$ be such that $d(y,-y)<r$ for all $y,y'\in Y$.
If there exists an odd map $X\to |\vr{Y}{r}|$, then there exists an odd function $f\colon X\to Y$ with modulus of discontinuity $\delta(f)\le r$.
\end{lemma}

\begin{proof}
Let $\sd(\vr{Y}{r})$ be the barycentric subdivision of $\vr{Y}{r}$, which means that a $k$-simplex of $\sd(\vr{Y}{r})$ is a chain of proper inclusions $\sigma_0 \subset \ldots \subset \sigma_k$, where each $\sigma_i$ is a simplex in $\vr{Y}{r}$.
A \emph{refinement} of $\sigma_0 \subset \ldots \subset \sigma_k$ is a chain of proper inclusions that contains each of $\sigma_0,\ldots,\sigma_k$ and potentially additional simplices of $\vr{Y}{r}$; such a refinement corresponds to a coface of $\sigma_0 \subset \ldots \subset \sigma_k$ in $\sd(\vr{Y}{r})$.
Two basic properties of barycentric subdivisions are that we have a natural bijection $|\vr{Y}{r}|=|\sd(\vr{Y}{r})|$ as sets, and that the interiors of the simplices of $\sd(\vr{Y}{r})$ form a partition of $|\sd(\vr{Y}{r})|$ (so long as we consider the interior of a vertex to be that vertex).

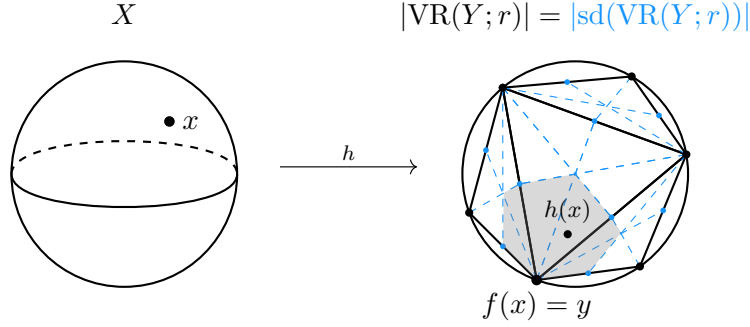
\begin{figure}[htb]
\centering
\begin{tikzcd}
    \begin{tikzpicture}
    \filldraw[fill=none, thick](0,0) circle (1.5);
    \draw [dashed, thick](-1.5,0) to [out=90,in=90,looseness=.5] (1.5,0);
    \draw [thick](1.5,0) to [out=270,in=270,looseness=.5] (-1.5,0);
    \node[fill=black,circle,inner sep=.13em] (x) at (0.6,0.7) {};
    \node[right,black] at (x)  {$x$};
    \node[fill=none] at (0,2){$X$};
    \end{tikzpicture}
    \hspace{1em}
    \ar[rr,"h"] 
    &&
    \hspace{-1.5em}
    \begin{tikzpicture}
    \filldraw[fill=none, thick](0,0) circle (1.5);
    \node[fill=none] at (0,2){$|\vr{Y}{r}|=\color{lightblue}{|\sd(\vr{Y}{r})|}$};
    \node[fill=none] (c) at (0,0){};
    \coordinate (1) at (10:1.5);
    \coordinate (2) at (130:1.5);
    \coordinate (3) at (250:1.5);
    \coordinate (4) at (60:1.5);
    \coordinate (5) at (200:1.5);
    \coordinate (6) at (305:1.5);
    \coordinate (12) at ($(1)!0.5!(2)$);
    \coordinate (23) at ($(2)!0.5!(3)$);
    \coordinate (13) at ($(1)!0.5!(3)$);
    
    \draw [color=black, thick] (1)--(2)--(4)--(1);
     \draw [color=black, thick] (3)--(2)--(5)--(3);
      \draw [color=black, thick] (1)--(3)--(6)--(1);
    \draw [color=lightblue,dashed] (1)--(23);
    \draw [color=lightblue,dashed] (2)--(13);
    \draw [color=lightblue,dashed] (3)--(12);
    \draw [color=lightblue,dashed] (4)--(12);
    \draw [color=lightblue,dashed] (5)--(23);
    \draw [color=lightblue,dashed] (6)--(13);
    \foreach \y in {4,6} 
        {\coordinate (1\y) at ($(1)!0.5!(\y)$);
        \node [fill=lightblue,circle,inner sep=.07em] at (1\y) {};}
    \foreach \y in {4,5} 
        {\coordinate (2\y) at ($(2)!0.5!(\y)$);
        \node [fill=lightblue,circle,inner sep=.07em] at (2\y) {};}
    \foreach \y in {5,6} 
        {\coordinate (3\y) at ($(3)!0.5!(\y)$);
        \node [fill=lightblue,circle,inner sep=.07em] at (3\y) {};}
    \draw [color=black, thick] (1)--(2)--(3)--(1);
    \draw [color=lightblue,dashed] (1)--(23);
    \draw [color=lightblue,dashed] (2)--(13);
    \draw [color=lightblue,dashed] (3)--(12);
    \draw [color=lightblue,dashed] (1)--(24);
    \draw [color=lightblue,dashed] (1)--(36);
    \draw [color=lightblue,dashed] (2)--(14);
    \draw [color=lightblue,dashed] (2)--(35);
    \draw [color=lightblue,dashed] (3)--(25);
    \draw [color=lightblue,dashed] (3)--(16);

    \fill[gray, opacity=.3] (3)-- (36)--(0.61,-0.78)-- (13)--(0,0)--(23)--(-0.95,-0.25)--(35)--(3);
    \node[fill=lightblue,circle,inner sep=.07em] at (12) {};    
    \node[fill=lightblue,circle,inner sep=.07em] at (23) {};
    \node[fill=lightblue,circle,inner sep=.07em] at (13) {};    
    \node[fill=black,circle,inner sep=.1em] at (1) {};
    \node[fill=black,circle,inner sep=.1em] at (2) {};
    \node[fill=black,circle,inner sep=.13em] at (3) {};
    \node[fill=black,circle,inner sep=.1em] at (4) {};
    \node[fill=black,circle,inner sep=.1em] at (5) {};
    \node[fill=black,circle,inner sep=.1em] at (6) {};
    \node[fill=black,circle,inner sep=.1em] (hx) at (-0.1,-0.8) {};
    \node[black, above] at (hx)  {\footnotesize $h(x)$};    
    \node[right,black, below] at (3)  {$f(x)=y$};
    \end{tikzpicture}
\end{tikzcd}
\caption{Figure accompanying the proof of Lemma~\ref{lem:VR-to-mod}, in the case when $X=S^2$ and $Y=S^1$.
Only four 2-simplices of $|\vr{Y}{r}|$ are drawn.
If $h(x)$ is any point in the gray shaded region, then necessarily $f(x)=y$.
}
\label{fig:modulus-discontinuity-tight}
\end{figure}

Let $h\colon X\to |\vr{Y}{r}|$ be an odd map.
Since $d(y,-y)<r$ for all $y,y'\in Y$, no simplex of $\vr{Y}{r}$ is equal to its antipode, by which we mean that if $\sigma=\{y_0,\ldots,y_m\}\in\vr{Y}{r}$, then $\sigma\neq-\sigma\coloneqq\{-y_0,\ldots,-y_m\}$.
Hence we can choose an odd function $v\colon \vr{Y}{r}\to Y$ that assigns each simplex $\sigma\in \vr{Y}{r}$ to a vertex of that simplex, where \emph{odd} means that $v(-\sigma)=-v(\sigma)$ for all $\sigma\in\vr{Y}{r}$.

Define the odd function $f\colon X\to Y$ as follows.
For $x\in X$, if $h(x)\in |\vr{Y}{r}|=|\sd(\vr{Y}{r})|$ is in the interior of the simplex $\sigma_0 \subset \ldots \subset \sigma_k$ in $\sd(\vr{Y}{r})$, then we define $f(x)=v(\sigma_0)$.
Since $h$ is odd and since $v$ is odd, it follows that $f$ is odd.
When one writes $h(x)$ using barycentric coordinates in $|\vr{Y}{r}|$, we note that $\sigma_0$ is the set of vertices in $Y$ achieving the largest barycentric coordinate of $h(x)$.
If there is a unique such vertex, then $f(x)$ is equal to this vertex, and otherwise ties are broken in an odd way using the function $v$; see Figure~\ref{fig:modulus-discontinuity-tight}.

We now show that the modulus of discontinuity of $f$ satisfies $\delta(f)\le r$.
Let $x\in X$.
Let $h(x)$ be in the interior of the simplex $\sigma_0 \subset \ldots \subset \sigma_k$ in $\sd(\vr{Y}{r})$.
Let $S\subseteq|\sd(\vr{Y}{r})|$ be the union of the interiors of all cofaces of $\sigma_0 \subset \ldots \subset \sigma_k$ in $\sd(\vr{Y}{r})$.
Then $S$ is an open set in $|\sd(\vr{Y}{r})|=|\vr{Y}{r}|$, and since $h$ is continuous, the preimage $h^{-1}(S)$ is an open neighborhood about $x$ in $X$.
We claim that $\diam(f(h^{-1}(S)))\le r$.
Indeed, let $x',x''\in h^{-1}(S)\subseteq X$.
By the definition of $S$, this means that $h(x')$ is in the interior of some simplex of $\sd(\vr{Y}{r})$ that is a coface of $\sigma_0 \subset \ldots \subset \sigma_k$, i.e.\ a refinement of $\sigma_0 \subset \ldots \subset \sigma_k$.
Hence $f(x')$ is a vertex of $\sigma_0$.
By the same argument, $f(x'')$ is a vertex of $\sigma_0$.
Since $\sigma_0$ is a simplex in $\vr{Y}{r}$, it follows that $d_{Y}(f(x'),f(x''))\le r$.
Because $x'$ and $x''$ are arbitrary points in $h^{-1}(S)$, we have shown that $\diam(f(h^{-1}(S)))\le r$.
Finally, since $h^{-1}(S)$ is an open neighborhood about $x$ in $X$, it follows that $\delta(f)\le r$.
\end{proof}

\begin{proof}[Proof of Theorem~\ref{thm:odd-modulus-discontinuity-bound-tight}]
Let $k \ge n$ and $\varepsilon>0$.
We must build an odd function $f\colon S^k \to S^n$ with modulus of discontinuity $\delta(f)\le c_{n,k}+\varepsilon$.

Theorem~2 of~\cite{ABF2} implies that $c_{n,k}<\pi$ for all $k\ge n$.
If $c_{n,k}+\varepsilon\ge \pi$, then one can simply let $f$ be an arbitrary odd function, since $\delta(f)\le\diam(S^n)=\pi$.
Hence it suffices to consider the case when $r\coloneqq c_{n,k}+\varepsilon<\pi$.

By the definition of $c_{n,k}$ in Definition~\ref{def:cnk}, there exists an odd map $h\colon S^k \to |\vr{S^n}{r}|$.
Since $r<\pi$, we have $d(y,-y)<r$ for all $y,y'\in S^n$.
So by Lemma~\ref{lem:VR-to-mod}, there exists an odd function $f\colon S^k\to S^n$ with modulus of discontinuity $\delta(f)\le r$.
\end{proof}

\begin{example}
We construct an explicit example in the case $n=1$.
In~\cite[Section~5]{ABF}, Adams, Bush, and Frick construct an injective odd map $\partial\cB_{2\ell+2}\hookrightarrow\vrm{S^1}{\frac{2\pi \ell}{2\ell+1}}$ from the boundary of the Barvinok--Novik orbitope to the Vietoris--Rips metric thickening of the circle.
The radial projection map $S^{2\ell+1}\to\partial\cB_{2\ell+2}$ in $\R^{2\ell+2}$ is a homeomorphism.
Furthermore, for any $\varepsilon>0$ we can use~\cite{AMMW,MoyMasters} to build an odd map $\vrm{S^1}{\frac{2\pi \ell}{2\ell+1}}\to|\vr{S^1}{\frac{2\pi \ell}{2\ell+1}+\varepsilon}|$; this follows from the $\varepsilon$-interleavings constructed in~\cite{AMMW,MoyMasters}, which in this setting can be made $\Z/2$ equivariant.
So by composition we obtain an odd map 
\[S^{2\ell+1}\to\partial\cB_{2\ell+2}\hookrightarrow\vrm{S^1}{\tfrac{2\pi \ell}{2\ell+1}}\to\left\vert\vr{S^1}{\tfrac{2\pi \ell}{2\ell+1}+\varepsilon}\right\vert.\]
By Lemma~\ref{lem:VR-to-mod}, we obtain an odd function $f\colon S^{2\ell+1} \to S^1$ with modulus of discontinuity $\delta(f)\le \frac{2\pi \ell}{2\ell+1}+\varepsilon=c_{1,2\ell+1}+\varepsilon$.
By restricting to the equator, we also obtain an odd function $f\colon S^{2\ell} \to S^1$ with $\delta(f)\le \frac{2\pi \ell}{2\ell+1}+\varepsilon=c_{1,2\ell}+\varepsilon$.
\end{example}

\subsection{Generalization of the lower bound on modulus of discontinuity}
\label{ssec:lower-bound-generalized}

We may generalize the domain $S^{k}$ in Theorem~\ref{thm:odd-modulus-discontinuity-bound} to any $\Z/2$ space $X$ with coindex at least $k$:

\begin{theorem}
\label{thm:odd-modulus-discontinuity-bound-generalized}
Let $X$ be a $\Z/2$ space with $\coind(X) \ge k$.
Then any odd function $f \colon X \to S^{n}$ with $k \ge n$ has modulus of discontinuity $\delta(f) \ge c_{n, k}$.
(In particular, this holds if $X$ is $(k - 1)$-connected.)
\end{theorem}

This follows since precomposing a function $f \colon X \to S^{n}$ with a map $g \colon S^{k} \to X$ cannot increase the modulus of discontinuity of $f$; more precisely, $\delta(f) \ge \delta(f \circ g)$.
In order to prove this fact (Lemma~\ref{lem:modulus-discontinuity-compose}), we define a pointwise version of the modulus of discontinuity.

Let $X$ be a topological space, let $Y$ be a metric space, and let $f \colon X \to Y$.
Then for $x \in X$, the \emph{modulus of discontinuity of $f$ at $x$} is
\begin{equation*}
\delta(f, x) := \inf\{\text{diam}(f(U)) \mid \text{$U$ is an open neighborhood of $x$}\}.
\end{equation*}
Note that $\delta(f) = \sup_{x \in X}\delta(f, x)$.
Then we have the following lemma:

\begin{lemma}
\label{lem:modulus-discontinuity-compose}
Let $X, Y$ be topological spaces, let $Z$ be a metric space, let $f \colon Y \to Z$ be a function, and let $g \colon X \to Y$ be a map.
Then
\begin{enumerate}
\item
For all $x \in X$, $\delta(f, g(x)) \ge \delta(f \circ g, x)$.
\item
$\delta(f) \ge \delta(f \circ g)$.
\end{enumerate}
\end{lemma}

\begin{proof}
For (1), let $U$ be an open neighborhood of $g(x)$ in $Y$.
Then $g^{-1}(U)$ is an open neighborhood of $x$ in $X$, and $(f \circ g)(g^{-1}(U)) \subseteq f(U)$.
Therefore,
$$\text{diam}(f(U)) \ge \text{diam}((f \circ g)(g^{-1}(U))) \ge \delta(f \circ g, x).$$
Taking the infimum over all such $U$, we obtain $\delta(f, g(x)) \ge \delta(f \circ g, x)$.
For (2), we have
$$\delta(f) = \sup_{y \in Y}\delta(f, y) \ge \sup_{x \in X}\delta(f, g(x)) \ge \sup_{x \in X}\delta(f \circ g, x) = \delta(f \circ g).$$
This completes the proof.
\end{proof}

\begin{proof}[Proof of Theorem~\ref{thm:odd-modulus-discontinuity-bound-generalized}]
Since $\coind(X) \ge k$, there exists an odd map $g \colon S^{k} \to X$.
Consider the composite function $f \circ g \colon S^{k} \to S^{n}$.
Since $g$ is a map, we have $\delta(f) \ge \delta(f \circ g)$, by Lemma~\ref{lem:modulus-discontinuity-compose}.
We also have $\delta(f \circ g) \ge c_{n, k}$, by Theorem~\ref{thm:odd-modulus-discontinuity-bound}.
Therefore, $\delta(f) \ge c_{n, k}$.
\end{proof}

\subsection{Lower bound on the Gromov--Hausdorff distance between Euclidean spheres}
\label{ssec:Euclidean}

In this subsection we use the Euclidean metric on spheres, instead of the geodesic metric.
For any integer $n\geq 0$, let $S_E^n$ denote the $n$-dimensional unit sphere equipped with the Euclidean metric; the Euclidean distance between $x,x' \in S_E^n$ is $\|x-x'\|$.
We will lower bound the distortion of odd maps $S_E^k \to S_E^n$ with $k>n$, and then use this to lower bound the Gromov--Hausdorff distance between $S_E^k$ and $S_E^n$.

We first note that the generalized Dubins--Schwarz inequality (Theorem~\ref{thm:odd-modulus-discontinuity-bound}) can be formulated in terms of the Euclidean metric on spheres.
The following result comes from combining Theorem~\ref{thm:odd-modulus-discontinuity-bound} with the fact that for any two points $x, x'$ on the unit sphere, we have $\|x - x'\| = 2\sin\left(\tfrac{d_{S^n}(x, x')}{2}\right)$.

\begin{corollary}\label{coro:euclidean_ds_inequality}
Any odd function $f\colon S_E^k \to S_E^n$ with $k \ge n$ has modulus of discontinuity $\delta(f)\geq 2\sin\left(\tfrac{c_{n,k}}{2}\right)$.
\end{corollary}

When $k= n+1$, the above corollary recovers the Dubins--Schwarz inequality in~\cite{dubins1981equidiscontinuity}, which states $\delta(f)\geq 2\sin\left(\tfrac{r_n}{2}\right)$ for maps $f\colon S_E^{n+1} \to S_E^n$.
In~\cite[Proposition~9.16]{lim2023gromov}, the Dubins--Schwarz inequality is combined with a Euclidean ``helmet trick'' (Lemma~\ref{lem:euclidean_helmet} below) to provide a lower bound on the Gromov--Hausdorff distance between the Euclidean spheres $S_E^{k}$ and $S_E^{n}$ for any integers $k>n\geq 1$.
This lower bound depends on $n$ but not on $k$, and so we instead use Corollary~\ref{coro:euclidean_ds_inequality} to obtain a lower bound depending on both $n$ and $k$.
Indeed, the following result recovers~\cite[Proposition~9.16]{lim2023gromov} when $k = n+1$ and obtains a refinement when $k> n+1$.

\begin{proposition}\label{prop:euclidean_gh_lower_bound}
For all integers $k\ge n$, we have $2\cdot d_\gh(S_E^n, S_E^k) \geq 2 - 2\cos\left(\frac{c_{n, k}}{2}\right)$.
\end{proposition}
Recall that Theorem~\ref{thm:proj-packings} implies $\lim_{k\to \infty} c_{n, k} = \pi$.
Meanwhile, it is shown in~\cite[Remark~9.1]{lim2023gromov} that $d_\gh(S_E^n, S_E^k)\leq 1$ for any integers $0\leq n< k$.
Therefore, the bound in the above proposition is asymptotically tight in the sense that $\lim_{k\to \infty}d_\gh(S_E^k, S_E^n) = 1$.
Our proof of Proposition~\ref{prop:euclidean_gh_lower_bound} will use the following ``helmet trick'' for Euclidean spheres.

\begin{lemma}[Lemma~9.14 in~\cite{lim2023gromov}]
\label{lem:euclidean_helmet}
For any $k, n \geq 0$, let the set $\varnothing \neq C \subseteq S_E^k$ satisfy $C \cap(-C)=\varnothing$, and let $\phi\colon C \to S_E^n$ be any function.
Then, the extension $\phi^* \colon C \cup(-C) \to S_E^n$ defined by
\[ \phi^*(x) = \begin{cases}
\phi(x) & \text{if }x\in C \\
-\phi(-x) & \text{otherwise}
\end{cases} \]
is odd and satisfies the distortion bound $\dis\left(\phi^*\right) \leqslant \sqrt{\dis(\phi)\left(4-\dis(\phi)\right)}$.
\end{lemma}

We are now ready to prove Proposition~\ref{prop:euclidean_gh_lower_bound}.

\begin{proof}
For any function $\phi\colon S_E^k \to S_E^{n}$, Lemma~\ref{lem:euclidean_helmet} guarantees that the existence of an odd function $\phi^*\colon S_E^k \to S_E^{n}$ such that $\dis\left(\phi^*\right) \leqslant \sqrt{\dis(\phi)\left(4-\dis(\phi)\right)}$.
We rearrange this bound by completing the square, using the fact that both $\dis(\phi)$ and $\dis(\phi^*)$ are bounded above by $2$ (the Euclidean diameter of unit spheres), obtaining $\dis(\phi)\geq 2 - \sqrt{4 - (\dis(\phi^\ast))^2}$.
As $0\leq \delta(\phi^\ast)\leq \dis(\phi^\ast)$, we have $\dis(\phi)\geq 2 - \sqrt{4 - (\delta(\phi^\ast))^2}$.
Thus, we use Corollary~\ref{coro:euclidean_ds_inequality} to obtain
\[
\dis(\phi) \geq 2 - \sqrt{4 - 4 \sin^2\left(\tfrac{c_{n, k}}{2}\right)} = 2- 2\sqrt{1- \sin^2\left(\tfrac{c_{n, k}}{2}\right)} =  2 - 2\cos\left(\tfrac{c_{n, k}}{2}\right)
\]
where in the last step we use the fact that $0\leq c_{n,k}\leq \pi$.
The result now follows from \eqref{eq:dgh}.
\end{proof}

\section{Conclusion and Questions}
\label{sec:conclusion}

As described in this article, the topology of Vietoris--Rips complexes of spheres are closely related to Gromov--Hausdorff distances between spheres, to approximate versions of the Borsuk--Ulam theorems for maps into lower- and higher-dimensional codomains, and to packings and coverings in projective space.
We expect similar relationships to be true for Vietoris--Rips complexes of other spaces.
We conclude this article with a list of open questions that we hope will attract interest from readers.

\begin{question}
\label{ques:upper}
There are not yet any known values of $n$ and $k$ where the lower bound 
$2\cdot d_{\gh}(S^n, S^k) \geq c_{n,k}$
is not an equality, and therefore it is reasonable to ask if this lower bound could be an equality in general.
In particular, can we find better functions $g\colon S^n \to S^k$ and $h\colon S^k \to S^n$ of low distortion and codistortion, providing upper bounds on the Gromov--Hausdorff distance between spheres?
By~\cite[Remark~1.1]{lim2023gromov}, it suffices to find a surjective function $S^k \to S^n$ of bounded distortion.

For example, can we find surjective functions $S^{2k} \to S^1$ and $S^{2k+1} \to S^1$ of distortion at most $\frac{2\pi k}{2k+1}$, in order to show that our lower bounds $2\cdot d_\gh(S^1,S^{2k})\ge \frac{2\pi k}{2k+1}$ and $2\cdot d_\gh(S^1,S^{2k+1})\ge \frac{2\pi k}{2k+1}$ are tight?
Can we find surjective functions $S^{n+1} \to S^n$ and $S^{n+2} \to S^n$ of distortion at most $r_n$, to show that the lower bounds $2\cdot d_\gh(S^n,S^{n+1})\ge r_n$ and $2\cdot d_\gh(S^n,S^{n+2})\ge r_n$ from~\cite{lim2023gromov} are tight?
In other words, we ask:
\begin{align*}
\text{Is\ \ }2\cdot d_{\gh}(S^1,S^{2k})=&\tfrac{2\pi k}{2k+1}=2\cdot d_{\gh}(S^1,S^{2k+1})? \\
\text{Is\ \ }2\cdot d_{\gh}(S^n,S^{n+1})=&\ \ r_n\ \ =2\cdot d_{\gh}(S^n,S^{n+2})?
\end{align*}
See~\cite{harrison2023quantitative,memoli2024embedding} for follow-up work on this question.
\end{question}

\begin{question}
\label{ques:other-direction}
In this paper, we have primarily thought of our Main Theorem and Theorem~\ref{thm:odd-modulus-discontinuity-bound} as providing lower bounds on the distortion of a function $f\colon S^k \to S^n$, or on the modulus of discontinuity of such a function, or on the Gromov--Hausdorff distance between $S^n$ and $S^k$, respectively.
However, one could also apply these results in the opposite direction in order to obtain new knowledge about Vietoris--Rips complexes of spheres.
Is it possible to find a function $f\colon S^k \to S^n$ of low distortion $r$ or low modulus of discontinuity $r$, in order to prove a new upper bound of the form $\coind(\vr{S^n}{r}) \ge k$ for a smaller value of $r$ than was previously known?
Similarly, is it possible to find functions showing $2\cdot d_\gh(S^n,S^k)$ is at most $r$ with $r$ small, in order to prove a new upper bound of the form $\coind(\vr{S^n}{r}) \ge k$?
If so, then~\cite[Theorem~3]{ABF2} would furthermore imply that there is no covering of $\RP^n$ by $k$ balls of radius $r$.
\end{question}

\begin{question}
Theorem~\ref{thm:odd-modulus-discontinuity-bound-tight} finds an odd function $f\colon S^k \to S^n$ with modulus of discontinuity arbitrarily close to $c_{n,k}$, showing that Theorem~\ref{thm:odd-modulus-discontinuity-bound} is tight.
Do there exist odd functions $f\colon S^k \to S^n$ with distortion arbitrarily close to $c_{n,k}$, which would show that the second inequality in our Main Theorem is tight?
\end{question}

\begin{question}
\label{ques:modulus-Euclidean}
Theorem~\ref{thm:odd-modulus-discontinuity-bound} lower bounds the modulus of discontinuity of odd functions $S^k \to S^n$ with $k>n$, and can be viewed as a discontinuous generalization of the Borsuk-Ulam theorem as stated in Theorem~\ref{thm:borsuk-ulam-odd}.
Can we obtain similar discontinuous generalizations of other equivalent formulations of the Borsuk-Ulam theorem, e.g., Theorem~\ref{thm:borsuk-ulam} and Theorem~\ref{thm:borsuk-ulam-alternate}?
Results in this direction will appear in an upcoming paper~\cite{adams2022quantifying}.
\end{question}

\begin{question}
\label{ques:GH-equivariance-Z/2}
Can we provide lower bounds on the Gromov--Hausdorff distances between more general families of $\Z/2$ metric spaces by obstructing the existence of equivariant maps to their Vietoris--Rips complexes?
In particular, can we generalize the sphere $S^{k}$ in our Main Theorem to a more general class of $(k - 1)$-connected $\Z/2$ spaces?

Theorem~\ref{thm:odd-modulus-discontinuity-bound-generalized} provides an answer to the analogous question for Theorem~\ref{thm:odd-modulus-discontinuity-bound}, our primary theorem bounding the modulus of discontinuity, replacing the $k$-sphere with a space of $Z/2$ coindex at least $k$.
Does an analogous generalization exist for 
the distortion bound in our Main Theorem?
\end{question}

\begin{question}
\label{ques:GH-equivariance-otherGroups}
Can one generalize our results to groups $G$ other than $G = \Z/2$?
As a start, one could attempt to extend our results from $\Z/2$ spaces, $\Z/2$ functions, and $\Z/2$ maps, to $\Z/p$ spaces, $\Z/p$ functions, and $\Z/p$ maps for other primes $p$.
Results in this direction will appear in upcoming papers~\cite{adams2022quantifying,lim2022GGHdist}.
\end{question}

\begin{question}
\label{ques:GH-projective}
Can our methods be used to obtain bounds on the Gromov--Hausdorff distances between spaces from families other than spheres, such as real projective space $\RP^n$, complex projective space $\CP^n$, ellipses with different eccentricity values, or ellipsoids of different dimensions or with different axis lengths?
What about Gromov--Hausdorff distances between spaces that each come from a different family, e.g.\ $d_\gh(\RP^n, S^n)$?
The first new homotopy type of the Vietoris--Rips metric thickening of $\RP^n$ is given in~\cite{AdamsHeimPeterson}.
Katz has studied the filling radius of complex projective spaces $\CP^n$~\cite{katz1983filling,katz9filling,katz1991neighborhoods,katz1991rational}, which is closely related to Vietoris--Rips complexes by~\cite{lim2020vietoris}.
The first new homotopy type of Vietoris--Rips complexes of small-eccentricity ellipses is given in~\cite{AAR}.
\end{question}

\begin{question}
\label{ques:finite}
What lower bound can we obtain for the Gromov-Hausdorff distances $d_{\gh}(X,S^k)$ and $d_{\gh}(X,Y)$, where $X$ and $Y$ are finite sets?
For example, what lower bounds can we get for the Gromov-Hausdorff distance $d_{\gh}(Q_n,S^k)$ and $d_\gh(Q_n,Q_k)$, where $Q_n$ is the vertex set of the hypercube graph, equipped with one of several natural choices of metric?
See~\cite{carlsson2020persistent,adamaszek2021vietoris,shukla2023vietoris} for recent papers on the Vietoris--Rips complexes of the hypercubes $Q_n$.
\end{question}

\begin{question}
\label{ques:DSgen-BU}
The Borsuk--Ulam theorem has many different corollaries.
Can our generalization of the Dubins--Schwarz inequality provide new generalizations of some of these corollaries?
Some results in this direction will appear in an upcoming paper~\cite{adams2022quantifying}.
\end{question}

\begin{question}
\label{ques:approx-Z2}
Let $X$ be a metric space that is ``approximately'' a $\Z/2$ space, by which we mean that acting by the generator twice returns a function that is not the identity map on the nose, but only close to being the identity map.
Is there a version of the Dubins--Schwarz inequality for ``approximate $\Z/2$ spheres''?
Can our machinery be adapted to provide bounds on the Gromov--Hausdorff distances between ``approximate'' $\Z/2$ metric spaces?
\end{question}

\begin{question}
Below we outline several fundamental open questions regarding the dependence of Gromov--Hausdorff distances between spheres on the dimensions of the spheres in question.
Several questions ask about upper bounds, which may be useful for working in the opposite direction of our results, i.e.\ upper bounds on Gromov--Hausdorff distances may be useful for drawing new conclusions about the homotopy connectivity of Vietoris--Rips complexes or packings and coverings in projective space.
We set $k>n$ in each question below.
\begin{itemize}
    \item
    We might already conjecture that $d_\gh(S^n,S^{n+1})=r_n=d_\gh(S^n,S^{n+2})$ based on existing bounds.
    However, Crabb~\cite[Theorem~3.1]{crabb2023borsuk} recently used characteristic classes and the cohomology of quotients of classical groups~\cite{baum1965cohomology} to prove that when $n+1$ is divisible by a large power of 2, then the diameter bound $r_n$ in~\cite[Theorem~3]{ABF} also works for maps $f\colon S^n\to\R^k$ into a selected number of higher dimensions $k$ with $k>n+2$.
    This motivates us to ask: Does there exist $n$ and $k > n+2$ so that $d_\gh(S^n,S^k)=r_n$?
    \item Is it always true that $d_\gh(S^n, S^k) \le d_\gh(S^n, S^{k+1})$? 
    That is, does $d_\gh$ increase monotonically as we move to the right in any row in Table~\ref{table:gh}?
    Note, the first row shows we \emph{cannot} guarantee a strict increase.
    This is Question I of Lim, Memoli, and Smith~\cite{lim2023gromov}.
    \item Is it always true that $d_\gh(S^n, S^k) \ge d_\gh(S^{n+1}, S^{k})$?
    That is, does $d_\gh$ decrease monotonically as we move downward in any column of Table~\ref{table:gh}?
    \item Is it always true that $d_\gh(S^n, S^k) \le d_\gh(S^{n+1}, S^{k+1})$?
    That is, does $d_\gh$ decrease monotonically as we follow any off-diagonal in Table~\ref{table:gh}?
    \item For fixed $m\ge 1$, is $2\cdot d_\gh(S^n, S^{n+m})$ bounded away from $\pi$ as $n\to\infty$?
    Theorem~\ref{thm:super-diag-upper-bound} gives an affirmative answer for $m=1$, and current bounds on $c_{n, n+m}$ (see~\cite[Corollary 3.2]{ABF2}) do not preclude an affirmative answer in general.
    \item For any $\delta>\frac{\pi}{2}$ and integer $m\ge 1$, does there exist a sufficiently large integer $N$ such that
    \[ 2\cdot d_\gh(S^n,S^k)\le \delta \text{ for all } n\ge N \text{ and } n<k\le n+m?\]
    Again, our lower bounds $c_{n,k}$ satisfy this property by~\cite[Corollary 3.2]{ABF2}.
\end{itemize}
\end{question}

\begin{question}
Definition~5.2.7 of~\cite{BushThesis} defines $s_{n,k}$ to be the infimal value of $r$ such that, for any odd map $f\colon S^n \to \R^k$ with $k\ge n$, there exists a subset $X\subseteq S^n$ of diameter at most $r$ such that the origin is in the convex hull of the image $f(X)\subseteq \R^k$.
See~\cite[Table on Page~80]{BushThesis} for the known values of $s_{n,k}$ and note the similarities with Table~\ref{table:gh}.
It is known that for $k\ge n$, we have $s_{n,k}\le c_{n,k}$.
Indeed, let $f\colon S^n \to \R^k$ be any odd map, and suppose $r > c_{n,k}$, i.e., suppose there is an odd map $S^k \to \vr{S^n}{r}$.
The map $f$ induces an odd map $\vr{S^n}{r}\to \R^k$.
By composition, we obtain an odd map $S^k \to \vr{S^n}{r}\to \R^k$.
We apply the standard Borsuk--Ulam theorem to this odd map $S^k\to\R^k$ to see that there is a point in $\vr{S^n}{r}$ that maps to the origin in $\R^k$, i.e.\ that there is a subset $X\subseteq S^n$ of diameter at most $r$ such that the origin is in the convex hull of $f(X)$.
Hence, $r \ge s_{n,k}$ and it follows that $s_{n,k}\le c_{n,k}$.
Could it be the case that $s_{n,k}=c_{n,k}$ for all $k\ge n$, or if not, what are the smallest values of $n$ and $k$ for which these quantities differ?
\end{question}

\bibliographystyle{abbrv}
\bibliography{GH-BU-VR.bib}

\end{document}